\documentclass[12pt]{amsart}

\usepackage{amsfonts, mathtools,latexsym, array, amssymb, }
\usepackage{mathrsfs}

\usepackage{psfrag}
\usepackage[all]{xy}
\usepackage{graphicx}
\usepackage{hyperref}
\usepackage{enumerate} 
\usepackage[usenames,dvipsnames]{color}
\newcommand{\comment}[1]{}

\numberwithin{equation}{section}
\theoremstyle{plain} 
\newtheorem*{maintheorem}{Main Theorem}
\newtheorem{theorem}[equation]{Theorem}

\newtheorem{corollary}[equation]{Corollary}
\newtheorem{lemma}[equation]{Lemma}
\newtheorem{fact}[equation]{Fact}

\newtheorem{proposition}[equation]{Proposition}
\newtheorem{claim}[equation]{Claim}
\newtheorem{example}[equation]{Example}

\newtheorem{definition}[equation]{Definition}

\newtheorem{remark}[equation]{Remark}
\newtheorem{remarks}[equation]{Remarks}

\newcommand\step[2]{\medbreak\noindent{\bf Step #1.} {\sl #2}\smallbreak}

\newcommand\alter[1]{\left\{\begin{array}{ll}#1\end{array}\right.}

\DeclarePairedDelimiter{\Size}{|}{|}

\newcommand\piS{\pi}

\newcommand\simN{{\scriptstyle \stackrel{N}\sim}}
\newcommand\BE{\approx}
\newcommand\cA{\mathcal A}
\newcommand\cB{\mathcal B}
\newcommand\cl[1]{\left<#1\right>_\sim}
\newcommand\cls[1]{\left<#1\right>} 
\newcommand\cG{\mathcal G}

\newcommand\cL{\mathcal L}

\newcommand\cO{\mathcal O}

\newcommand\Diff{\operatorname{Diff}}

\newcommand\eps{\epsilon}
\newcommand\id{\operatorname{id}}
\newcommand\NN{{\mathbb N^0}}
\newcommand\RR{{\mathbb R}}

\newcommand\ZZ{{\mathbb Z}}
\newcommand\per{\operatorname{per}}
\newcommand\pih{{\hat\pi}}

\newcommand\Prob{{\operatorname{Prob}}}
\newcommand\Proberg{\operatorname{Prob}_{\operatorname{erg}}}

\newcommand\rec{{\rm rec}}
\newcommand\reg{\ast}

\newcommand\BEpi{\operatorname{\stackrel\pi\BE}}
\newcommand\simpi{\operatorname{\stackrel\pi\sim}}

\newcommand\hchi{\widehat{\chi}}
\newcommand\Probhyp{\Prob_{\operatorname{hyp}}}

\renewcommand\top{{\operatorname{top}}}

\begin{document}

\title[Degrees of Bowen factors]{The degree of Bowen factors and injective codings of diffeomorphisms}


\author{J\'er\^ome Buzzi}
\address{Laboratoire de Math\'ematiques d'Orsay - Universit\'e Paris-Sud}
\email{jerome.buzzi@math.u-psud.fr}

\begin{date}
{\today}
\end{date}

\begin{abstract}
We show that symbolic finite-to-one extensions of the type constructed by O.~Sarig for surface diffeomorphisms induce H\"older-continuous \emph{conjugacies} on large sets. 
We deduce this from their \emph{Bowen property}. This notion, introduced in a joint work with M.\ Boyle, generalizes a fact first observed by R.\ Bowen for Markov partitions. We rely on the notion of degree from finite equivalence theory and magic word isomorphisms.

As an application, we give lower bounds on the number of periodic points first for surface diffeomorphisms (improving a result of Sarig) and for Sina\"\i\ billiards maps (building on a result of Baladi and Demers). Finally we characterize surface diffeomorphisms admitting a H\"older-continuous coding of all their aperiodic hyperbolic measures and give a slightly weaker construction preserving local compactness.
\end{abstract}

{{\renewcommand{\thefootnote}{}%
}
{\renewcommand{\thefootnote}{}%
}
 
\maketitle

\section{Introduction}\label{s-intro}

In this text, a dynamical system  is an automorphism of a standard Borel space   and all measures are understood to be ergodic, invariant, Borel probability measures.  A Markov shift is a ``subshift defined by a countable oriented graph''. It is equipped with a standard distance (see Sec. 2 for precise definitions and references).

For  a smooth diffeomorphism $f$ of a compact manifold,  a measure is called \emph{hyperbolic} if it has no zero Lyapunov exponent and both a positive and a negative exponent (see, e.g., \cite[Chap.~S]{Katok-Hasselblatt} for background on smooth ergodic theory). A measure is called \emph{$\chi$-hyperbolic} for some $\chi>0$, if it is hyperbolic and has no exponent in the interval $[-\chi,\chi]$. 

\medbreak

We build conjugacies from the finite-to-one extensions of surface diffeomorphisms of Sarig \cite{Sarig-JAMS} making them injective while preserving the H\"older-continuity and discarding only a subset negligible with respect to all (invariant probability) measures:

\begin{theorem}\label{theorem-SurfDiffCoding}
Let $f$ be a diffeomorphism with H\"older-continuous differential mapping a compact boundaryless $C^\infty$ surface $M$ to itself. For  any numbers $0<\chi'<\chi$, there exist a Markov shift $S:X\to X$ and a H\"older continuous map $\pi:X\to M$ such that:
 \begin{itemize}
  \item $\pi\circ S=f\circ\pi$;
  \item $\pi:X\to M$ is injective;
  \item $\pi(X)$ has full measure for any $\chi$-hyperbolic measure;
   \item for any periodic $x\in X$, the periodic point $\pi(x)$ defines a $\chi'$-hyperbolic measure.
 \end{itemize}
\end{theorem}

Previous injectivity results \cite{BoyleBuzzi-2017,Sarig-Private,BCS} were only with respect to a single measure, a restricted class of measures, or by jettisoning the continuity and discarding periodic orbits. 

As an application, we deduce from well-known results on Markov shifts  estimates on the periodic counts of surface diffeomorphisms.
Consider the hyperbolic periodic points with given \emph{minimal} period and Lyapunov exponents (defined by identifying a periodic orbit with the obvious measure) bounded away from zero by a number $\chi>0$:
 $$
    \per_\chi(f,n):=\{x\in M:\{f^k(x):k\in\ZZ\} \text{ has cardinality $n$ and is $\chi$-hyperbolic}\}.
 $$

We denote the cardinality of a set by $|\cdot|$.

\begin{theorem}\label{theorem-SurfDiffPeriodic}
Let $f$ be a $C^\infty$-diffeomorphism of a closed surface $M$. Assume that its topological entropy $h_\top(f)$ is positive. Then there is some integer $p\geq1$ such that:
  \begin{equation}\label{eq-period-lower}
   \forall\chi<h_\top(f)\quad \liminf_{\scriptsize\begin{matrix} n\to\infty\\ p|n\end{matrix}} e^{-n\cdot h_\top(f)} \Size{\per_\chi(f,n)} \geq p.
 \end{equation}
 If the diffeomorphism is topologically mixing, one can take $p=1$.
\end{theorem}

This improves the previous estimate due to Sarig \cite{Sarig-JAMS}:
 \begin{equation}\label{eq-Sarig-lower}
   \exists p\geq1\; \liminf_{\scriptsize\begin{matrix} n\to\infty\\ p|n\end{matrix}} e^{-n\cdot h_\top(f)} \Size{\{x:x=f^nx \text{ and is $\chi$-hyperbolic}\}} >0.
 \end{equation}
Indeed, not only do we have an explicit constant, but we control the minimal period. By comparison, the estimate \eqref{eq-Sarig-lower} is compatible, e.g., with $\per_\chi(f,n)=\emptyset$ for infinitely many $n$ a multiple of $p$.
\medbreak

Thanks to works of  Baladi and Demers~\cite{Baladi-Demers-2018} and Lima and Matheus~\cite{Lima-Matheus-2016}, our general results can be applied to the classical collision map $T_B$ of any two-dimensional Sina\"\i\ billiard $B$ (see, e.g., \cite{Chernov-book}  for background) satisfying the following two conditions: 
\begin{itemize}
 \item[(BD1)] all trajectories have a nontangential collision (see  \cite[strong finite horizon property before Rem. 1.1]{Baladi-Demers-2018});
 \item[(BD2)] some combinatorial entropy (denoted by $h_*$ and introduced in \cite[Def.~2.1]{Baladi-Demers-2018}) is above some threshold defined in \cite[eqs.\ (1.2)~and~(1.2)]{Baladi-Demers-2018}.
\end{itemize}
We denote by $\Lambda_B$ the hyperbolicity constant  of $T_B$ from eqs. (2.2)-(2.3) in \cite{Baladi-Demers-2018}. We remark that Baladi and Demers also prove that $h_*=\sup\{h(T_B,\nu):\nu\in\Proberg(T_B)\}$.

\begin{theorem}\label{theorem-BilliardMap}
If $T_B$ is the collision map  of a two-dimensional Sina\"\i\ billiard $B$ satisfying conditions (BD1) and (BD2), then:
 $$
  \liminf_{n\to\infty} e^{-n\cdot h_*} \Size{\per_{\Lambda_B}(T_B,n)} \geq 1.
 $$
\end{theorem}

This strengthens \cite[Cor. 2.7]{Baladi-Demers-2018} by eliminating the possibility of a period, counting the periodic orbits by their minimal periods, and replacing the positive lower bound by the integer $1$.

\medbreak

We derive these results by proving a general result about a large class of symbolic dynamics, see our Main Theorem below.

\medbreak

Theorem~\ref{theorem-SurfDiffCoding} improves on Sarig's coding by making it injective. One would also like to have an image as large as possible. The following shows that, in some sense, one cannot much improve on Sarig's result in this direction:

\begin{theorem}\label{prop-Holder-opti}
Let $f\in\Diff^r(M)$ be a diffeomorphism of a closed surface with $r>1$. Then there exist  a Markov shift $S:X\to X$ and a  map $\pi:(S,X)\to (f,M)$ such that $f\circ\pi=\pi\circ S$ and:
\begin{enumerate}[(i)]
  \item for all $\mu\in\Proberg(f)$ with positive entropy, there is $\nu\in\Prob(S)$ with $\pi_*(\nu)=\mu$;
 \item $\pi$ is H\"older-continuous for the standard metric on $X$;
\end{enumerate}
if and only if the exponents of ergodic invariant probability measures with positive entropy are bounded away from zero.
\end{theorem}

\subsection{General theorem}

The above will be a consequence of an abstract theorem about factors of Markov shifts. A \emph{symbolic system} $(S,X)$ is  some shift-invariant subset of $\cA^\ZZ$ where $\cA$, the \emph{alphabet},  is a countable (possibly finite) set, together with the action of the shift $S:(x_n)_{n\in\ZZ}\mapsto(x_{n+1})_{n\in\ZZ}$. We equip $X$ with the standard distance: $d(x,y):=\exp(-\inf\{|n|:x_n\ne y_n\}$).
A \emph{semiconjugacy} $\pi:(S,X)\to(T,Y)$ between dynamical systems $(S,X)$ and $(T,Y)$ is a map $\pi:X\to Y$ such that $\pi\circ S=T\circ\pi$ and $\pi(X)\subset Y$.
A {\it one-block code} is a semiconjugacy $\pi:X\to Y$ between symbolic systems  such that $\pi(x)=(\Pi(x_n))_{n\in\ZZ}$ for some map  between the alphabets (this map $\Pi:\cA\to\cB$ is the \emph{code}).
A  map $\phi$ between two metric spaces is \emph{$1$-Lipschitz} if $d(\phi(x),\phi(y))\leq d(x,y)$ for all pairs of points $x,y$.

Recall the following definition from \cite{BoyleBuzzi-2017}, taken from R. Bowen's analysis of Markov partitions \cite{Bowen-AxiomA}. 

\begin{definition}[Boyle-Buzzi]\label{def-Bowen}
A  semiconjugacy $\pi:(S,X)\to (T,Y)$ satisfies the {\bf Bowen property}  if $X$ is symbolic and if there is a reflexive and symmetric relation~$\sim$ on its alphabet $\cA$ such that, for all $x,y\in X$,
 \begin{equation}\label{eq-Bowen}
  \pi(x) = \pi(y) \iff \forall n\in\ZZ\;\; x_n\sim y_n,
 \end{equation}
in which case we  say that  $x,y\in X$ are \emph{Bowen equivalent} and write $x\BE y$.

The relation $\sim$ on $\cA$ is called a \emph{Bowen relation} for $\pi$ (or admitted by $\pi$). It is said to be \emph{locally finite} if $\{b\in\cA:b\sim a\}$ is finite for each $a\in\cA$.
\end{definition}

We note that the Bowen property generalizes both:
 \begin{itemize}
  \item[-] D. Fried's \emph{finitely presented systems} \cite{Fried-1987} which are exactly the continuous Bowen factors of subshifts of finite type;
  \item[-] \emph{one-block codes} which admits as a transitive Borel relations the equivalence relations defined by their codes. There is a partial converse: any Bowen semiconjugacy $\pi$ with a transitive Bowen relation $\sim$ can be written as $\pi=\psi\circ\Pi$ where $\Pi$ is the one-block code defined by $a\mapsto \{b:b\sim a\}$ and $\psi:\Pi(X)\to\pi(X)$ is a Borel conjugacy.
 \end{itemize}

\medbreak

We show:

\begin{maintheorem}\label{maintheorem}
Let  $(S,X)$ be a Markov shift on some alphabet $\cA$. Let $X^\#$ be its \emph{regular part}, i.e., the set of sequences $x\in X$ such that, for some $u,v\in\cA$, 
 $$
   \text{ $u$ occurs infinitely many times in $(x_n)_{n\leq0}$ and $v$ occurs infinitely many times in $(x_n)_{n\geq0}$}.
 $$
Let $\pi:(S,X^\#)\to(T,Y)$ be a Borel semiconjugacy such that:
\begin{itemize}
  \item[-] $(T,Y)$ is a Borel automorphism;
  \item[-] $\pi$ is finite-to-one, i.e., $\pi^{-1}(y)$ is finite for every $y\in Y$;
  \item[-] $\pi$ has the Bowen property with respect to a locally finite relation on $\cA$.
\end{itemize}   
Then there are a Markov shift $(\tilde S,\tilde X)$ and a $1$-Lipschitz map $\phi:\tilde X\to X^\#$ such that $\pi\circ\phi:\tilde X\to Y$ defines an injective  semiconjugacy and $\pi\circ\phi(\tilde X^\#)$ carries all invariant measures of $\pi(X^\#)$. 
\end{maintheorem}

Remark that if $\pi$ is continuous or H\"older-continuous, then so is $\pi\circ\phi$.

\subsection{Further results, comments, and questions}
Note that a map $\phi:\tilde X\to X$  is $1$-Lipschitz if and only if there is a length-preserving map $\Phi:\bigcup_{n\geq0}\cL_{2n+1}(\tilde X)\to \bigcup_{n\geq0}\cL_{2n+1}(X)$ such that $\phi(x)_{[-n,n]}=\Phi(x_{[-n,n]})$ for all $n\geq0$. The $1$-Lipschitz map in the above theorem  may fail to be a one-block code because it may fail to commute with the shift.

\subsubsection*{Ingredients of the proof} 
We adapt classical ideas from the theory of subshifts of finite type. The first part builds on tools from finite equivalence theory and more specifically joint work with Mike Boyle \cite{BoyleBuzzi-2017}. This leads to Theorem~\ref{thm-Bowen-quotient} which is an abstract version of an unpublished result of Sarig \cite{Sarig-Private}. The second part of the proof involves ideas from magic word isomorphisms and the degree of almost conjugacies \cite[chap. 9]{LindMarcus-book}. It allows to partition according to the number of preimages while preserving the Markov structure. We conclude by injectively coding subsets with larger and larger numbers of preimages.

\subsubsection*{Good coding for given measures}
We can specialize our results to a given measure of interest. For instance, given a surface diffeomorphism with positive topological entropy and a distinguished ergodic measure maximizing the entropy $\mu$, we obtain an irreducible Markov shift $X$ and a H\"older-continuous conjugacy $\pi:X\to M$ such that $\pi(X)$ has full $\mu$-measure. This was implicit in \cite[Prop. 6.3]{BoyleBuzzi-2017}. We refer to \cite{BCS} for further results in this direction.

\subsubsection*{Bounds for periodic points}
Kaloshin \cite{Kaloshin-SuperExp} has shown that, $C^r$-generically ($1\leq r<\infty$), the number of periodic points grows arbitrarily fast with the period. However, these periodic points have Lyapunov exponents going to zero. In fact, Burguet~\cite{Burguet-PeriodicPoints} has shown the following logarithmic estimate, for any $C^\infty$ surface diffeomorphism:
 $$
  \forall\chi<h_\top(f)\;  \lim_{n\to\infty, p|n} \frac1n\log|\per_\chi(f,n)| = h_\top(f).
 $$

\medbreak
\noindent{\bf Question 1.} {\it Is there a $C^\infty$ surface diffeomorphism $f$ with positive entropy such that, for some $\chi>0$:}
 $$
     \limsup_{n\to\infty} e^{-n\cdot h_\top(f)} \Size{\per_\chi(f,n)} =\infty\; ?
  $$
  
\subsubsection*{Beyond surface diffeomorphisms}
Ben Ovadia's higher-dimensional generalization \cite{benovadia-coding} of Sarig's coding also yields finite-to-one semiconjugacies that are Bowen with respect to a locally finite relation. Hence our abstract theorem also applies in this setting. 

\subsubsection*{Better symbolic representations}
For a topologically transitive surface diffeomorphism, S.~Crovisier, O.~Sarig, and the author \cite{BCS} have shown that, for any given parameter $\chi>0$,  there is a finite-to-one, H\"older-continuous \emph{transitive} symbolic dynamics coding a subset carrying all $\chi$-hyperbolic measures. Applying our main theorem makes this coding injective but destroys the transitivity.  We ask:

\medbreak
\noindent{\bf Question 2.} {\it For a topologically transitive $C^\infty$ diffeomorphism of a closed surface and any number $\chi>0$, can one  get a H\"older-continuous injective coding by a \emph{transitive} Markov shift a subset carrying all $\chi$-hyperbolic measures?}
\medbreak

Our Theorems \ref{thm-ae-injectivity}~and~\ref{thm-large-injectivityset} below provide partial solutions. We build a H\"older-continuous, finite-to-one coding by a transitive Markov shift whose injectivity set is ``large'' in a weaker sense than above: it has full measure with respect to a given measure or for all fully supported measures.
These theorems are applied to surface diffeomorphisms in \cite{BCS}.
\medbreak

To capture all hyperbolic measures, one can apply Sarig's construction countably many times with a parameter $\chi$ decreasing to $0$. One obtains a sequence of semiconjugacies with larger and larger images but smaller and smaller H\"older exponents. In \cite{BoyleBuzzi-2017}, together with M.~Boyle, we were able to "fuse" all these semiconjugacies by using a Borel construction.

\medbreak
\noindent{\bf Question 3.} {\it Given a surface diffeomorphism, can one get a \emph{continuous} finite-to-one coding by a Markov shift of a subset carrying \emph{all hyperbolic measures}? Can it be done injectively?}
\medbreak

Because of Theorem~\ref{prop-Holder-opti} one cannot ask for a H\"older-continuous semiconjugacy.

\subsubsection*{Local compactness}
Our proof of the Main Theorem does not preserve local compactness. We do not know if it can be done.  The following very natural question asked by the referee remains open:

\medbreak
\noindent{\bf Question 4.} {\it In Theorem~\ref{maintheorem}, is it possible to code using a locally compact Markov shift?}
\medbreak

In Appendix~\ref{appendix-Q}, we provide a partial answer: we obtain local compactness but  injectivity holds only after restricting to the regular part.

\subsection{Outline of the paper}
In Section~\ref{s-definitions}, we recall some basic definitions and make some comments about Bowen relations. In Section~\ref{s-bowen}, we introduce Bowen quotients inspired by a classical construction of Manning \cite{Manning-1971} from the theory of Markov partitions. These are an abstract version of a construction of Sarig \cite{Sarig-Private}. The proof rests on Proposition \ref{prop-resolving}, which adapts lemmas from the theory of finite equivalence of subshifts of finite type due to Hedlund \cite{Hedlund-1969} and Coven and Paul \cite{Coven-Paul-1974,Coven-Paul-1975,Coven-Paul-1977}. 
In Section~\ref{s-magic}, we combinatorially characterize the fibers with minimal cardinality by adapting the notion of magic word from the theory of almost conjugacy of shifts of finite type  (see \cite[chap.~9]{LindMarcus-book}).

In Section~\ref{s-main}, we show how to a get a coding with a large injectivity set, especially with respect to a given measure or to all fully supported measures and then deduce the main theorem from the previous constructions. We proceed by induction on the number of preimages. The Bowen quotients make the semiconjugacy injective where its fibers had a given  cardinality and then discard these points. The magic word theory preserves the Markov structure and the Bowen property.

In Section~\ref{s-applications}, we apply our main theorem to Sarig's coding of surface diffeomorphisms \cite{Sarig-JAMS} and prove Theorems \ref{theorem-SurfDiffCoding}~and~\ref{theorem-SurfDiffPeriodic} using Newhouse \cite{Newhouse-1989} (for smoothness) and Buzzi-Crovisier-Sarig \cite{BCS} (for transitivity). 
In Section~\ref{s-obstruction}, we prove Theorem~\ref{prop-Holder-opti} characterizing surface diffeomorphisms with H\"older-continuous codings. In the Appendices we further discuss the Bowen relation, provide a locally compact construction, and deduce Theorem~\ref{theorem-BilliardMap}.

\bigbreak

\noindent{\bf Acknowledgments.} I thank Pierre Berger, Sylvain Crovisier, Yuri Lima, and Omri Sarig for  useful comments. I especially thank  Mike Boyle for pointing out mistakes in an early version of this text and simplifying the proof of Proposition~\ref{prop-resolving}. I am also grateful to Sylvain Crovisier for discussions leading to Section~\ref{s-injectivity} and to Yuntao Zang for pointing out an imprecision in Lemma~\ref{lem-barchi1}. I finally thank the referee for both corrections and suggestions that have improved this paper.

\section{Definitions and first properties}\label{s-definitions}

\subsection{Borel systems}\label{s-Borel}

A standard Borel space is a set equipped with the Borel $\sigma$-field generated by a Polish topology (i.e., generated by a  metric making the space complete and separable, see \cite{Kechris-book} for background).
 A \emph{dynamical system}  is an automorphism $S$ of such a space $X$. We denote it by $(S,X)$ (or just $S$ or $X$ when convenient).

\medbreak

A \emph{full subset}  for $S$ is a subset of $X$ with  measure equal to $1$ for all measures\footnote{Recall that all measures in this paper are understood to be ergodic and invariant Borel probability measures.} of $S$. A subset is \emph{null} if its complement is a full subset.
We say that a property of points holds \emph{almost everywhere} (or just a.e.) without reference to a measure, if it holds on such a full subset.

\medbreak

By the Lusin-Novikov Theorem \cite[(18.10)]{Kechris-book},  the direct image of a Polish space by a countable-to-one Borel map is Borel. In fact, there is a countable partition of the Polish space into Borel subsets on which the map is injective and one can apply the Lusin-Suslin Theorem \cite[(15.2)]{Kechris-book}. Therefore:

\begin{lemma}\label{lem-null-sets}
Let $p:(S,X)\to(T,Y)$ be a Borel semiconjugacy between dynamical systems. If $\nu\in\Prob(T)$ satisfies $p^{-1}(y)$ is finite and nonempty for $\nu$-a.e. $y\in Y$, then there is $\mu\in\Prob(S)$ with $p_*(\mu)=\nu$.
In particular, if $p$ is finite-to-one and onto:
 \begin{itemize}
  \item[-]  $p_*:\Prob(S)\to\Prob(T)$ is onto;
  \item[-]  for any Borel subsets $U\subset X$, $V\subset Y$: $U$ is a null subset $\iff$ $p(U)$ is a null subset; $V$ is a null subset $\iff$ $p^{-1}(V)$ is a null subset.
  \end{itemize} 
\end{lemma}

\subsection{Symbolic dynamics}\label{s-symbolic}
Let  $\cA$ be a countable (possibly finite) discrete set, called the alphabet. The shift is $(x_n)_{n\in\ZZ}\mapsto(x_{n+1})_{n\in\ZZ}$. A \emph{symbolic system} $(S,X)$ is the restriction $S$ of the shift to some invariant subset $X$ of $\cA^\ZZ$ with the usual distance: $d(x,y):=\exp\left(-\inf\{|n|:x_n\ne y_n\}\right)$. A symbolic system needs not be closed.

For $x\in X$ and integers $a\leq b$, $x_{[a,b]}=x_ax_{a+1}\dots x_b\in\cA^{b-a+1}$ and $x_{[a,b)}=x_{[a,b-1]}$.
The set of \emph{$X,n$-words} is $\cL_n(X):=\{x_{[0,n)}:x\in X\}$ and the language is $\cL(X):=\bigcup_{n\geq0} \cL_n(X)$. A word $w\in\cL_n(X)$ has length $|w|:=n$. The word $w$ \emph{occurs} at $n\in\ZZ$ in some $x\in X$ if $x_{[n,n+|w|)}=w$. An $n$-word $w$ defines a cylinder:
 $$
     [w]_X := \{y\in X:y_{[0,n)}=w\}.
 $$
As usual, words differing by an integer translation of their indices are identified.

Sarig's \emph{regular set} is the subset $X^\#\subset X$ of sequences $x\in X$ such that there are $u,v\in\cA$ satisfying:
 $$
     \{n\geq0:x_{-n}=u\},\; \{n\geq0:x_n=v\} \text{ are both infinite}.
 $$
If the alphabet is finite, then $X^\#=X$. If $\pi:X\to Y$ is a semiconjugacy with $X$ a symbolic system, its \emph{regular part} is the restriction $\pi^\#:X^\#\to Y$. We also write $\pi_\#$ or $\pi$ when convenient.

A sequence $x\in X$ is \emph{word recurrent}  if any word $w$ that occurs in $x$ occurs infinitely often in both $x_{(-\infty,0]}$ and $x_{[0,\infty)}$. We say that $w$ \emph{occurs i.o. in} $x$ or that $x$ \emph{sees i.o. $w$}.\label{def-sees-io} We denote by $X^\rec\subset X$ the set of such sequences. Note that it carries all invariant probability measures on $X$ (in particular, it contains all periodic orbits). We have the obvious inclusion $X^\rec\subset X^\#$.

\medbreak

A \emph{Markov shift} $(S,X)$ is a symbolic system over some alphabet $\cA$ such that $X$ can be characterized as the set of bi-infinite paths on some simple, directed graph\footnote{A simple directed graph is a graph with oriented arrows and such that for any vertices $a,b$ there is at most one arrow from $a$ to $b$.} $\cG$, that is, 
 $$
    X := \{x\in\cA^\ZZ:\forall n\in\ZZ\; x_n\stackrel{\cG}\to x_{n+1}\}.
 $$
The graph $\cG$  \emph{describes} the Markov shift $X$.
 
\subsection{Remarks about the Bowen property} We comment on related notions, the (non)\-uniqueness of the symmetric relation involved in its characterization and its eventual extension from the regular part to the whole of a factor.

\smallbreak

\noindent
1) A semiconjugacy may have the Bowen property without being Borel. Indeed, if $\pi$ is Bowen, then so is $\varphi\circ\pi$ for any self-conjugacy $\varphi:Y\to Y$ (i.e., a bijection that commutes with the dynamics). 

\smallbreak\noindent
2) Being Bowen and finite-to-one are independent properties of semiconjugacies (neither implies the other) as shown by the following examples: (i) $\pi:\{0,1,2\}^\ZZ\to\{0,1\}^\ZZ$ given by the block code $a\mapsto b\mod 2$ is Bowen and infinite-to-one; (ii) $\pi:\{0,1\}^\ZZ\to\mathbb S^1$ given by $\pi(x)=\exp 2i\pi\sum_{n\geq0} 2^{-n-1}x_n$ which is at most 2-to-1 but cannot be Bowen since it is neither injective nor constant.

\smallbreak\noindent
3) On the one hand, many Bowen semiconjugacies do not admit any transitive Bowen relations.\footnote{For instance, a continuous Bowen factor of a compact symbolic system with a transitive Bowen relation must be zero dimensional.} On the other hand, the equivalence between sequences must be transitive. Thus Bowen relations are special reflexive and symmetric relations on the alphabet. Bowen \cite[p.~13]{Bowen-AxiomA} asked:

\medbreak

{\sc Problem (Bowen).} {\it Let  $A,B$ be two $n\times n$-matrices with entries zero or one. Let  $\BE$ be the relation defined on $X:=\{x\in\{1,\dots,n\}^\ZZ:\forall p\in\ZZ$ $A(x_p,x_{p+1})=1\}$ by $x\BE y\iff\forall n\in\ZZ$ $B(x_n,y_n)=1$. Decide whether this relation is transitive. If so, decide whether the shift on $X/\BE$ is topologically conjugate to the non-wandering set of a uniformly hyperbolic system.}

\smallbreak\noindent
4) A given Bowen semiconjugacy can admit distinct Bowen relations. See Appendix~\ref{appendix-Bowen} for the canonical relation defined by a semiconjugacy.

\smallbreak\noindent
5) In our main examples the Bowen property hold in the regular part of the symbolic system. Even when the semiconjugacy has a unique uniformly continuous extension to the whole symbolic system, the latter may fail to satisfy the Bowen property, see Appendix~\ref{appendix-Bowen}.

\section{Bowen Quotients}\label{s-bowen}

We introduce our basic construction: the Bowen quotient. 
Given an integer $N\geq1$  and a first semiconjugacy $\pi$, we are going to build another semiconjugacy $\pi_N$ whose preimages are the sets of $N$ preimages of a common point. This is a purely combinatorial construction which:
 \begin{itemize}
  \item[--] preserves the class of finite-to-one, Bowen semiconjugacies of regular parts of Markov shifts;
  \item[--] produces $\pi_N$ which is one-to-one above the points where the first semiconjugacy was $N$-to-$1$;
  \item[--] has in its image the points with at least $N$ preimages by $\pi$, up to a null set.
 \end{itemize}
This construction is closely related to previous work with Boyle  \cite[Prop. 6.3]{BoyleBuzzi-2017}.
Similar constructions go back to Hedlund \cite{Hedlund-1969} and Coven and Paul \cite{Coven-Paul-1974,Coven-Paul-1975,Coven-Paul-1977} (for subshifts of finite type, see \cite[chap. 8 and 9]{LindMarcus-book}); Manning \cite{Manning-1971}  and Bowen \cite{Bowen-AxiomA}  (for coding of Axiom-A diffeomorphisms).

\subsection{Definition and statement}

Let $(S,X)$ be a Markov shift with alphabet $\cA$ and underlying graph $\cG$. Let $\piS:X^\#\to Y$ be a Borel semiconjugacy on the regular part $X^\#$. Assume that it admits a Bowen relation~$\sim$. Given an integer $N\geq1$, let:
 \begin{itemize}
  \item[-] $\cA_N$ be the collection of subsets $A\subset\cA$ with cardinality $N$ whose elements are pairwise related by $\sim$;
  \item[-] $\cG_N$ be the simple directed graph over $\cA_N$ with arrows: $A\stackrel{\cG_N}{\to}B$ if and only if there is a bijection $\phi:A\to B$ such that $a\stackrel{\cG}{\to}b\iff b=\phi(a)$ for all $(a,b)\in A\times B$.
\end{itemize}

\begin{definition}\label{def-Bowen-quotient}
Let $\pi:X^\#\to Y$ be a semiconugacy with a locally finite Bowen relation~$\sim$. Given an integer $N\geq1$, the \emph{Bowen quotient of order $N$} of $(\piS:X^\#\to Y,\sim)$ is $(\piS_N:X_N^\#\to Y,\stackrel N\sim)$ where:
\begin{itemize}
  \item[(BQ1)] $X_N^\#$ is the regular part of the Markov shift $X_N$ defined by the graph $\cG_N$;
   \item[(BQ2)] $\piS_N:X_N^\#\to Y$ satisfies for all $\hat x\in X_N^\#$, $\piS_N(\hat x)=\pi(x)$ for any $x\in X^\#$ s.t.  $x_n\in\hat x_n$ {\rm (}$\forall n\in\ZZ${\rm )};
     \item[(BQ3)] $A\stackrel N\sim B$ if and only if $\forall (a,b)\in A\times B\; a\sim b$.
 \end{itemize}
\end{definition}

The following definition will be convenient for our purposes.

\begin{definition}\label{def-excellent}
A semiconjugacy $\pi:X\to Y$ is  \emph{excellent} for some relation $\sim$ if $X$ is a symbolic system, $Y$ is a dynamical system and:
 \begin{itemize}
  \item[(EX1)] $\sim$ is a locally finite, reflexive and symmetric relation on the alphabet of $X$;
  \item[(EX2)] $\pi$ is Bowen with respect to the relation $\sim$;
  \item[(EX3)] $\pi$ is Borel and finite-to-one.
\end{itemize}  
\end{definition}

In the next statement and throughout this paper, $\binom{p}{q}=\frac{p!}{q!(q-p)!}$ and is zero if $q>p$.

\begin{theorem}\label{thm-Bowen-quotient}
Let $X$ be a Markov shift with regular part $X^\#$.
Let $\piS:X^\#\to Y$ be an excellent semiconjugacy for some Bowen relation $\sim$. 
Then, for any integer  $N\geq 1$, the Bowen quotient $(\piS_N:X_N^\#\to Y,\stackrel N\sim)$ of order $N$ is well-defined and excellent.

\medbreak

Moreover there is a finite-to-one, $1$-Lipschitz map $q_N:X_N\to X$  such that:
 \begin{enumerate}
  \item if $X$ is locally compact, then so is $X_N$;
  \item $\pi_N=\pi\circ q_N:X^\#\to Y$;
  \item $|\piS_N^{-1}(y)|\leq \textstyle\binom{|\piS^{-1}(y)|}{N}$ for all $y\in Y$ with equality except on a null set;
  \item $q_N:X_N^\#\to X^\#$ is proper, i.e., for any compact set $K\subset X^\#$, $q_N^{-1}(K)\cap X_N^\#$ is compact.
 \end{enumerate}

\end{theorem}

\begin{definition}\label{def-deg-spectrum}
The \emph{degree spectrum} of a Borel semiconjugacy $\pi:Z\to Y$ is:
 $$
    \Delta(\pi):=\{k\geq1:\{y\in Y:|\pi^{-1}(y)|=k\}\text{ is not a null set}\}.
 $$
\end{definition}

\begin{corollary}\label{c-degspec-N}
In the setting of the above theorem,  $\pi_N(X_N^\#)\subset \pi(X^\#)$. In fact,
 \begin{equation}\label{eq-Delta-PiN}
   \Delta(\pi_N) = \left\{\binom{r}{N}:r\in\Delta(\pi)\text{ and }r\geq N\right\} 
 \end{equation}
 and, for each $r\in\Delta(\pi)$ with $r\geq N$,
   $
     \bigl\{ y\in Y: |\pi_N^{-1}(y)| = \binom{r}{N} \bigr\} $ is contained in
          $\bigl\{ y\in Y: |\pi^{-1}(y)| \geq r \bigr\}$ 
          and equal to $\bigl\{ y\in Y: |\pi^{-1}(y)| = r \bigr\}$ up to a null set.
\end{corollary} 

All the claims above are straightforward consequences of the theorem, except possibly for eq.~\eqref{eq-Delta-PiN}, which we now prove. Let $Z:=\{y\in Y: |\pi_N^{-1}(y)|\ne\binom{|\pi^{-1}(y)|}{N}\}$. It is a null set by item~(3) above. If $r\in\Delta(\pi)$ and $r\geq N$, let $s:=\binom{|\pi^{-1}(y)|}{N}\geq 1$. Note that $\{y\in Y:|\pi^{-1}(y)|=r\}\setminus Z$ is not a null set and is included in  $\{y\in Y:|\pi_N^{-1}(y)|=s\}$, so $s\in\Delta(\pi_N)$. For the converse, let $s\in\Delta(\pi_N)$ so $\{y\in Y:|\pi_N^{-1}(y)|=s\}\setminus Z$ is not a null set. Hence, by item (3), $s=\binom{r}{N}$ for some $r\geq N$ such that $\{y\in Y:|\pi^{-1}(y)|=r\}$ is not a null set. Hence $r\in\Delta(\pi)$. Eq.~\eqref{eq-Delta-PiN} is proved. 

\begin{remarks}\label{rem-Bowen-Q}$ $

1. The proof of the theorem will give an explicit null set $Y_0$ where the inequality in item (3) may be strict.

2. The following example shows that   $X_N$ may fail to be irreducible even if $X$ is irreducible.  However, using the magic word theory of Section~\ref{s-magic} we will show in Theorem~\ref{thm-large-injectivityset} that one can restrict the semiconjugacy to an irreducible component of $X_N$ without diminishing the image.
\end{remarks}

{\footnotesize
\begin{example}
Let $\cG_0$ be a simple, directed graph with set of vertices $\cA_0$ and let $p$ be a positive integer. Define $\cG$ over $\cA:=\cA_0\times(\ZZ/p\ZZ)$ by: $(a,i)\stackrel{\cG}\to(b,j)\stackrel{\operatorname{def}}\iff i+1=j$ and $a\stackrel{\cG_0}\to b$. Define a symmetric relation $\sim$ on $\cA$ by $(a,i)\sim(b,j)\stackrel{\operatorname{def}}\iff a=b$ and consider a Bowen semiconjugacy with this relation. Hence, for any $N\geq1$, $\cA_N=\{ \{a\}\times I:a\in\cA_0$ and $I\subset\ZZ/p\ZZ$ with $|I|=N\}$ and $\{a\}\times I\to\{b\}\times J$ if and only if $a\to b$ and $I+1=J$.

Observe that $\cG_N$ may fail to be irreducible even when $\cG$ is irreducible. For instance, if $p=4$ and $N=2$, for any $a\in\cG_0$, $\{(a,0),(a,1)\}$ and $\{(a,0),(a,2)\}$ belong to distinct irreducible components of $\cG_N$.
\end{example}
}

\subsection{Resolving property}

We begin by studying the combinatorics of finite fibers over an orbit with some recurrence. 
Let $(S,X)$ be a Markov shift with alphabet $\mathcal A$ and $(T,Y)$ be some dynamical system. Let $\piS:X^\#\to T$ be a semiconjugacy with some Bowen relation $\sim$.   Denote the Bowen equivalence by $\BE$ and the equivalence class of $x\in X^\#$ by
  $$
  \cls{x}:=\{y \in X^\#:y\BE x\}.
  $$

Let us call a point $x\in X$ \emph{recurrent for some function} $\phi:X\to\ZZ$ if, for each $n\in\ZZ$, $\{k\in\ZZ:\phi(S^kx)=\phi(S^nx)\}$ is neither lower bounded nor upper bounded. Poincar\'e recurrence implies that the set of recurrent points for any given measurable function is a full set.

\begin{proposition}\label{prop-resolving}
Let $(\piS,\sim)$ be an excellent semiconjugacy defined on a Markov shift $X$ described by some graph $\mathcal G$.
For $x\in X$ and $n\in\ZZ$, let
 $
   \cA(x,n):=\{y_n:y\in\cls{x}\}.
 $ 
Let $X^\reg$ be the set of points $x\in X^\#$ which are simultaneously recurrent for the three following functions:
 $$
  R(x):=\Size{\cA(x,0)},\quad 
  n_+(x):=\Size{\{y_{[0,\infty)}:y\in\cls{x}\}},\quad
  n_-(x):=\Size{\{y_{(-\infty,0]}:y\in\cls{x}\}}.
 $$
Then for any $x\in X^\reg$ and $n\in\ZZ$,
 \begin{enumerate}[  (a)]
  \item
  for each $n\in\ZZ$, the restriction $x\mapsto x_n$ defines a bijection between $\cls{x}$ and $\cA(x,n)$;
  \item
   for each $n\in\ZZ$, $|\cA(x,n)|=|\cA(x,n+1)|$ and, for every $a\in\cA(x,n)$ there is a unique $b\in\cA(x,n+1)$ such that $a\stackrel{\mathcal G}\to b$;
  \item $\cls{x}\subset X^\reg$.
\end{enumerate}

\end{proposition}

This proposition is related to finite equivalence theory and especially some classical results of Coven and Paul \cite{Coven-Paul-1974} (see \cite[Thm 8.1.16]{LindMarcus-book}). 

\begin{proof}$ $

\step{1}{Given $x\in X^\#$ recurrent for the function $n_+$ and $a\in\ZZ$, there is an integer $\ell\geq1$ (which can be taken arbitrarily large) such that, among $y\in\cls{x}$, $y_{[a,a+\ell)}$ determines $y$.}

In the above situation, we say that  $[a,a+\ell)$ is an \emph{admissible interval} for $x$.

\medbreak

Let  $x,a$ be as above. Since $\cls{x}$ is finite, $n_+(S^{-n}x)=\Size{\cls{x}}$ for $n$ large enough. Observe that the function $n_+$ is monotone along orbits. By recurrence, it is constant along the orbits.  Hence $\Size{\cls{x}}=n_+(S^ax)=\Size{\{y_{[a,a+\ell)}:y\in\cls{x}\}}$ for all large integers~$\ell$. Fixing such an integer  $\ell$, the obvious surjectivity of the map $y\mapsto y_{[a,a+\ell)}$ implies its bijectivity.

\step{2}{Item (a): for any $n\in\ZZ$, $y_n$ determines $y$ for $y\in\cls{x}$.}

Given $n\in\ZZ$, Step 1 provides an admissible interval $[a,b]\subset[n+1,\infty)$. If there were distinct $y,y'\in\cls{x}$ such that $y_n=y'_n$ but $y_{(-\infty,n)}\ne y'_{(-\infty,n)}$, the spliced sequence $y'_{(-\infty,n)}y_{[n,\infty)}$ of $X^\#$ coinciding with $y$ on $[a,b]$, would contradict the admissibility of $[a,b]$. Thus $y_n$ determines $y_{(-\infty,n]}$ for $y\in\cls{x}$.

Symmetric arguments (using $n_-$) show that $y_n$ determines the whole sequence $y\in\cls{x}$, proving item (a).

\step{3}{Item (b): $\mathcal R:=\{(a,b)\in\cA(x,n)\times\cA(x,n+1):a\stackrel{\mathcal G}\to b\}$ is a bijection}

Observe first that for every $a\in\cA(x,n)$, $a=y_n$ for some $y\in\cls{x}$ so that $a\to y_{n+1}$ with $y_{n+1}\in\cA(x,n+1)$.
Assume  $a\stackrel{\mathcal G}\to b$ and $a\stackrel{\mathcal G}\to b'$ with $a\in\cA(x,n)$ and $b,b'\in\cA(x,n+1)$. Therefore there are $y,z,z'\in\cls{x}$ such that $y_n=a$, $z_{n+1}=b$, and $z'_{n+1}=b'$. Considering the splicings $y_{(-\infty,n]}z_{[n+1,\infty)}$ and  $y_{(-\infty,n]}z'_{[n+1,\infty)}$, Step 2 implies that $b=b'$. Thus $\mathcal R$ defines a unique map $\cA(x,n)\to\cA(x,n+1)$. 
A symmetric argument gives an inverse map $\cA(x,n+1)\to\cA(x,n)$, hence $\mathcal R$ is bijective: item (b) is proved.

\step{4}{Item (c): if $x\in X^\reg$, then  $\cls{x}\subset X^\reg$}

This is clear from the definition of $X^\reg$. The proposition is proved.
\end{proof}

\subsection{Bowen quotients}
We prove Theorem \ref{thm-Bowen-quotient}. To begin with, we let $\cG_N,\cA_N,X_N$, and $\stackrel N\sim$ as in Definition~\ref{def-Bowen-quotient}. Items (BQ1) and (BQ3) are then satisfied by construction. 

We now define $\pi_N$ to satisfy (BQ2) and item (2) of the theorem. We note the following easy consequence of the definition of $\cG_N$:

\begin{fact}\label{fact-lift-quotient}
For $-\infty\leq i<0<j\leq\infty$, let $\hat x=(\hat x_n)_{i<n<j}$ be a finite or infinite  path on $\cG_N$. For each $a\in\hat x_0$, there is a unique path $Q(\hat x,a):=(x^a_n)_{i<n<j}$ on $\cG$ such that $x^a_0=a$ and $x^a_n\in\hat x^a_n$ for all $i<n<j$. 
\end{fact}

For convenience we select some total order on the alphabet of $X$.
We define $q_N:X_N\to X$ by setting $q_N(\hat x):=Q(\hat x,\min(\hat x_0))$. Note that $q_N$ is $1$-Lipschitz since the $X$-word $x_{-n}\dots x_n$  in Fact~\ref{fact-lift-quotient} depends only on  $\hat x_{-n}\dots\hat x_n$. If $\hat x\in X_N^\#$, then $x_n\in\hat x_n$ for all $n\in\ZZ$ implies $x\in X^\#$.

In particular, the following defines a Borel map:
 $$
 \pi_N : X_N^\#\to Y,\quad \hat x\mapsto\pi\circ q_N(\hat x).
 $$ 
This map satisfies (BQ2) in Definition \ref{def-Bowen-quotient}: for any $\hat x\in X_N^\#$, any $x\in X^\#$ with $x_n\in\hat x_n$ for all $n\in\ZZ$, $\pi_N(\hat x)=\pi(x)$. 
Indeed, the condition $x_n\in\hat x_n$ implies that $x\BE q_N(\hat x)$ so $\pi_N(\hat x)=\pi(x)$, proving (BQ2) as well as item (2) in the theorem.

The map $\pi_N$ is a semiconjugacy. Indeed, $\pi_N(S_N(\hat x))=\pi(y)$ with  $y_n\in\hat x_{n+1}$ and $T(\pi_N(\hat x))=T(\pi(z))=\pi(S(z))$ with $z_{n}\in\hat x_{n}$. Hence $y\BE S(z)$ so $\pi(y)=\pi(S(z))$, and
 $
   \pi_N\circ S_N=T\circ\pi_N$.

 Note that (BQ2) ensures the uniqueness of such semiconjugacy. Indeed, if $\pi_N,\pi_N'$ are two such maps, then given any $\hat x\in X_N^\#$, $\pi_N(\hat x)=\pi(x)$ and $\pi_N'(\hat x)=\pi(y)$ where $x_n,y_n\in\hat x_n$ for each $n$. Hence $x\BE y$ and so $\pi(x)=\pi(y)$: $\pi_N=\pi_N'$. 
 
 \medbreak
 
 It remains to show   items (1), (3), and (4) and excellency, i.e., items (EX1)-(EX4).

 \medbreak

Let us check that this construction preserves the local compactness. Indeed,  for any $A\in\cA_N$, take $a\in A$ and observe that if $A\stackrel N\to B$, then $B\subset \{c:\exists b\; a\to b\text{ and }b\sim c\}$ which is finite, so $A$ has finite outdegree if $a$ has. A similar argument applies to the indegree. Thus $X_N$ is locally compact if $X$ is, proving item (1).

A similar argument shows that $\stackrel N\sim$ is locally finite. Observe also that $\stackrel N\sim$ is reflexive and symmetric. Thus item (EX1)  is proved.

To bound the number of preimages under $\pi_N$, let $y\in Y$ and write $\piS^{-1}(y)=\{x^1,\dots,x^r\}$ with $r=|\piS^{-1}(y)|$. By construction, any preimage under $\pi_N$ corresponds to a set of $N$ preimages under $\pi$: 
 \begin{equation}\label{eq-preimages}
    \pi_N^{-1}(y)\subset \{ \hat x^J:=(\{x^j_n:j\in J\})_{n\in\ZZ}:J\subset\{1,\dots,r\},\; |J|=N\}.
  \end{equation}

Therefore $|\piS_N^{-1}(y)|\leq\binom{r}{N}$. In particular,                                                                                                                                                                                                                                                                                                                                                                                                                                                                                                                                                                                                                                                                                                                                                                                                                                                                                                                                                                                                                                                                                                              $\piS_N$ and therefore $q_N|X_N^\#$ are finite-to-one. Note that (EX3) is established.
\medbreak

We let
 $$
   Y_0:=\piS(X^\#\setminus X^\reg) \subset Y
  $$
where $X^\reg$ is the good set defined by Proposition \ref{prop-resolving}. From Lemma~\ref{lem-null-sets}, $Y_0$ is a null set for $T$ since $X^\reg$ is a full set for $S$. For $y\in \pi(X^\#)\setminus Y_0$, Proposition~\ref{prop-resolving} implies that the sequences $\hat x^J$ in eq.~\eqref{eq-preimages} belong in $X_N$. In fact, they must belong to $X_N^\#$.  Indeed, since $x^1\in X^\#$, there is some $a\in\cA$ such that $x^1_n=a$ for infinitely many $n\geq0$. For those indices $n$, $\hat x^J_n$ is contained in the finite set $\{b:b\sim a\}$, hence must take some value infinitely many times. The same holds for negative indices, proving that $\hat x^J\in X_N^\#$. Thus  the inclusion in eq.~\eqref{eq-preimages} is an equality and:
 $$
    \forall y\in\pi(X^\#)\setminus Y_0\;\; |\piS_N^{-1}(y)| = \binom{r}{N}. 
 $$
Item (3) of the theorem is proved. We note for future reference the following consequence:
\begin{fact}\label{fact-special-lift}
Let $\pi:X^\#\to Y$ be an excellent semiconjugacy with Bowen quotient $\pi_N:X_N^\#\to Y$.
For all $x\in X^\#$ outside a null set, if $|\pi^{-1}(\pi(x))|\geq N$ then $\exists \hat x\in X_N^\#$ s.t. $\forall n\in\ZZ\; x_n\in\hat x_n$.
 \end{fact}

\medbreak

We check that $\stackrel N\sim$ is a Bowel relation for $\pi_N$. First, let $\hat x,\hat y\in X_N^\#$ with $\pi_N(\hat x)=\pi_N(\hat y)$.  Let $a\in\hat x_0$ and $b\in\hat y_0$. Fact.~\ref{fact-lift-quotient} gives (unique) sequences $x\in[a]_X^\#$, $y\in[b]_X^\#$ with $\pi(x)=\pi_N(\hat x)$, $\pi_N(\hat y)=\pi(y)$. Thus $\pi(x)=\pi(y)$ and  $x_0\sim y_0$. It follows that $\hat x_0\stackrel{N}\sim\hat y_0$ and then $\hat x\stackrel{N}\BE \hat y$, by equivariance. 

Conversely, let $\hat x,\hat y\in X_N^\#$ with $\hat x\stackrel{N}\BE\hat y$. Picking $a\in\hat x_0$ and $b\in\hat y_0$, Fact~\ref{fact-lift-quotient} gives $x\in[a]_X^\#$, $y\in[b]_X^\#$ such that $x_n\in\hat x_n$ and $y_n\in\hat y_n$ for all $n\in\ZZ$. Thus $\pi(x)=\pi_N(\hat x)$ and $\pi_N(\hat y)=\pi(y)$. From the definition of $\stackrel{N}\sim$, we have $x\BE y$. The Bowen property for $\sim$ implies $\pi(x)=\pi(y)$ hence $\pi_N(\hat x)=\pi_N(\hat y)$. The Bowen property (EX2) is established.

\medbreak

Finally, we prove that $q_N$ is proper. Note that a subset $K$ of a symbolic system is relatively compact if and only if, for each $n\in\ZZ$, $\{x_n:x\in K\}$ is finite. Fix a relatively compact $K\subset X^\#$ and $n\in\ZZ$. By construction, $q_N(\hat x)\in K$ implies that $\hat x_n$, a set of $N$ symbols from $\cA$,   contains only symbols that are Bowen related to elements of $\{x_n:x\in K\}$. Since $\sim$ is locally finite, it follows that $\{\hat x_n:q_N(\hat x)\in K\}$ is finite and $q^{-1}(K)$ is relatively compact. Item (4) is proved. \hfill\qed

\section{Combinatorial degree}\label{s-magic}

We are going to characterize the subset of a Bowen semiconjugacy where the cardinality of the fibers is minimal by the recurrence of some words. To this end, we adapt the notions of degree and magic word from the classical theory of one-block codes between subshifts of finite type  (see Hedlund \cite{Hedlund-1969} and more generally \cite[chap. 9]{LindMarcus-book}). 

\medbreak

It is convenient to disregard the factor map $\pi:X\to Y$ and to  focus on the symbolic system $X$ and the Bowen relation. 

\begin{definition}
Given a symbolic system  $(S,Z)$ on some alphabet $\cA$, an \emph{(abstract) Bowen relation} is a reflexive, symmetric relation $\sim$ on $\cA$ such that the relation on $Z$ defined by $x\BE y\stackrel{\textrm{def}}{\iff}\forall n\in\ZZ\; x_n\sim y_n$   is an equivalence relation. 
\end{definition}

In this section, $\sim$ is a Bowen relation on the regular part $X^\#$ of a Markov shift $X$. Recall  that $X^\rec\subset X^\#$ is the set of word recurrent sequences in $X$ (i.e., any word that occurs once is seen i.o. --see p.~\pageref{def-sees-io}). 
Recall also that the Bowen equivalence class of any $x\in X^\#$ is denoted by: 
$$\cls{x}:=\{y\in X^\#:y\BE x\}.$$  

\subsection{Degree of Bowen relations}

The relation $\sim$ on the alphabet of $X^\#$ induces  another reflexive and symmetric relation on $\cL(X^\#)$ (also denoted by $\sim$) according to $v\sim w\stackrel{\rm def}\iff|v|=|w|$ and $v_1\sim w_1,\dots,v_{|v|-1}\sim w_{|v|-1}$.

We will consider the languages $\cL(X^\rec)\subset\cL(X^\#)\subset\cL(X)$. In general, they are distinct. However they are equal when $X$ is the disjoint union of its irreducible components.

\newcommand\degree{\delta}

\begin{definition}
Given a Bowen relation $\sim$, the \emph{degree} of  a word $w\in\cL(X)$  at some index $0\leq i<|w|$  is:
 $$\begin{aligned}
   &\degree_\sim(w,i):=\Size{\{v_i:v\in\cL(X^\#),\; v\sim w\}}\\
    &\degree_\sim(w):=\min\{\degree_\sim(w,i): 0\leq i<|w|\} .
 \end{aligned}$$
The degree of $\sim$ is:
 $$
    \degree_\rec(\sim):=\inf\{\degree_\sim(w):w\in\cL(X^\rec)\}.
 $$
A  \emph{magic word} is a word $w\in\cL(X^\rec)$ realizing this infimum. A couple $(w,i)$ that realizes it is called a magic couple. 
\end{definition}

Observe that for any $w\in\cL(X^\rec)$, $\deg_\sim(w)\geq1$ (since $\sim$ is reflexive). As soon as $\deg_\rec(\sim)$ is finite (e.g., if $\sim$ is locally finite),   there always exist magic words.

Given a word $W\in\cL(X^\#)$,  $X_W$ denotes the set of sequences that see i.o. $W$:
 $$
    X_W:=\{x\in X:\exists m_k,n_k\to\infty \text{ such that $W$ occurs in $x$ at $-m_k$ and at }n_k\}.
 $$
Note that $X_W$ is an invariant, possibly empty, subset of $X^\#$.
We start with two simple lemmas.

\begin{lemma}\label{lem-magic-injec}
Assume that  the  Bowen equivalence classes:
 $
    \cls{x}:=\{y\in X^\#:y\BE x\} $ are finite for all $x\in X^\#$.
Let  $W\in\cL(X^\#)$ with $\degree_\sim(W)=1$. If $x\in X^\#$ sees i.o. $W$, then $x$ is only equivalent to itself, that is:
 $$
    \forall x\in X_W \quad \cls{x}=\{x\}.
 $$
\end{lemma}

\begin{proof}
Let $0\leq I<|W|$ such that $\delta_\sim(W,I)=1$ and $x\in X_W$. Pick an increasing sequence of integers $(n_k)_{k\in\ZZ}$ such that $x_{[n_k,n_k+|W|)}=W$. If there is a distinct $y\in X^\#$ with $x\BE y$, one can find $k<l$ such that $x_{[n_k+I,n_l+I]}\ne y_{[n_k+I,n_l+I]}$.  However, $y_{[n_k,n_k+|W|)}\sim W$ implies that $y_{n_k+I}=W_I$ since $\delta_\sim(W,I)=1$. Consider the infinitely many distinct arbitrary concatenations of the two words $x_{[n_k+I,n_l+I)},y_{[n_k+I,n_l+I)}$. They  belong to the Markov shift $X$ and in fact to $X^\#$ since they see i.o. the symbol $W_I$. Moreover, they belong to a single Bowen equivalence class which is infinite, a contradiction.
\end{proof}

\begin{lemma}\label{lem-deg-less-fiber}
Assume that $\sim$ is a locally finite Bowen relation. For any $x\in X^\#$,  
  \begin{equation}\label{eq-deg-less-fiber}
   |\cls{x}|\geq\degree_\sim(x):=\min\{\delta_\sim(x_p\dots x_{p+\ell-1}):p\in\ZZ,\; \ell\geq1\}.
  \end{equation}
\end{lemma}

\begin{proof}
Fix $x\in X^\#$. For each $n\geq1$, let $\cA_n:=\{y_0:y\in\cls{x_{[-n,n]}}\}$. This defines a non-increasing sequence of sets contained in $\{b\in\cA:b\sim x_0\}$ which is finite. Hence there is $n_0$ such that $\cA_n=\cA_{n_0}$  for $n\geq n_0$. Note that $|\cA_{n_0}|\geq \degree_\sim(x)$. Now, fix $a\in\cA_{n_0}$ and, for each $n\geq0$, pick $y^n\in\cls{x_{[-n,n]}}$ with $y^n_0=a$. For each $k\in\ZZ$, $\{y^n_k:n\geq|k|\}$ is finite (since the Bowen relation is locally finite). Thus one can find an accumulation point $y\in\cA^\ZZ$ (i.e., there is $n_j\uparrow\infty$ such that, for each $k\in\ZZ$, $y_k=y^{n_j}_k$ for all large $j$). It is easy to check that $y\in X^\#$ and $y\in\cls{x}$. Varying $a\in\cA_{n_0}$, eq.~\eqref{eq-deg-less-fiber} follows.
\end{proof}

\subsection{Magic semiconjugacies}

We relate the combinatorial degree of the Bowen relation with the cardinality of the fibers of the semiconjugacy.

\begin{theorem}\label{thm-degree}
Let $X$ be a Markov shift and let $\sim$ be a locally finite Bowen relation for $X^\#$. Let $x\in X^\rec$ with $\cls{x}$ finite. The following are equivalent:
 \begin{enumerate}[\quad(a)]
  \item $\cls{x}$ has exactly $\degree_\rec(\sim)$ elements;
  \item $x$ sees some magic word $W$ for $\sim$.
 \end{enumerate}
\end{theorem}

The following example shows that  the implication $(b)\implies(a)$ in Theorem~\ref{thm-degree} may fail when $\cls{x}$ is infinite or when $x\notin X^\rec$.

{\footnotesize
\begin{example}
Let $X=\{0,1,2\}^\ZZ$ and for $a,b\in\{0,1,2\}$, let $a\sim b\!\!\iff\!\! |b-a|=0,2$. Note that $X^\#=X$, $\cL(X^\rec)=\cL(X)$, and $\degree_\rec(\sim)=\degree_\sim(1)=1$. For $x\in X^\rec$ distinct from $1^\infty$ such as $x=(10)^\infty$, $\cls{x}$ is infinite. For $y=1^\infty0^k1^\infty$ with $k\geq1$, $y\in X^\#\setminus X^\rec$ sees i.o. the magic word $1$, however:  $\Size{\cls{y}}=2^k>\degree_\rec(\sim)$.
\end{example}
}

The degree has a geometric meaning:

\begin{corollary}\label{cor-Xmin}
Let $X$ be a Markov shift such that $X^\rec\ne\emptyset$ and let $\sim$ be a locally finite Bowen relation for $X^\#$. If $\cls{x}$ is finite for each $x\in X^\rec$,  then
 \begin{equation}\label{eq-deg-fiber}
    \degree_\rec(\sim) = \min\{\Size{\cls{x}}:x\in X^\rec\}
     =\min\{k\geq1:\{x\in X^\#:|\cls{x}|=k\}\text{ is not null}\}.
 \end{equation}
In particular, $\degree_\rec(\sim)$ only depends on the Bowen equivalence relation $\BE$. 
\end{corollary}

\begin{proof}
The inequality $\degree_\rec(\sim) \leq \min\{\Size{\cls{x}}:x\in X^\rec\}$ follows from Lemma~\ref{lem-deg-less-fiber} since $\degree_\rec(\sim)\leq\degree_\sim(x)$ for $x\in X^\rec$. Conversely, let $W$ be a magic word for $\sim$ over $X^\rec$. By definition, there is $x^0\in X^\rec$ that sees i.o. $W$. By Theorem~\ref{thm-degree}, $|\cls{x^0}|=\degree_\rec(\sim)$, proving the first equality.

We show that $\degree_\rec(\sim)$ is equal to $d:=\min\{k\geq1:\{x\in X^\#:|\cls{x}|=k\}\text{ is not null}\}$.  Since $\{x\in X^\#:|\cls{x}|=d\}$ has positive measure for some invariant probability measure, it contains a recurrent point, so $\degree_\rec(\sim)=\min\{|\cls{x}|:x\in X^\rec\}\leq d$. 

Conversely, there is $x\in X^\rec$ such that $|\cls{x}|=\degree_\rec(\sim)$. By Theorem~\ref{thm-degree}, $x$ sees i.o. some magic word $W$.  Since $X$ is a Markov shift, one can find $y\in X^\#$ that sees i.o. $W$ and is periodic. Its orbit is a non-null set, hence $d\leq\degree_\rec(\sim)$.
\end{proof}

\begin{remark}
The next example shows that there is no simple analogue of Corollary~\ref{cor-Xmin} for $X^\#$,  even if one replaces the degree $\deg_\rec(\sim)$ by $\min\{\deg_\sim(w):w\in \cL(X^\#)\}$.
\end{remark}

\begin{figure}
\includegraphics[width=0.6\textwidth]{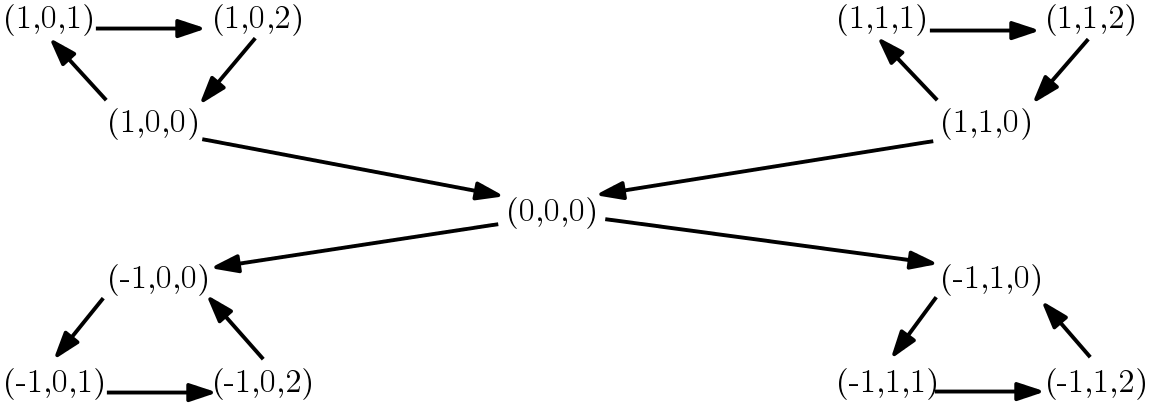}
\caption{The subshift of finite type in Example \ref{exBadRegular}}\label{fig-Sigma}
\end{figure}

{\footnotesize
\begin{example}\label{exBadRegular}
Let $X$ be the  subshift of finite type defined by the directed graph in Fig.~\ref{fig-Sigma}.
Define $\pi:X\to\{-1,0,+1\}^\ZZ$ as the projection on the first coordinate with Bowen relation $(a,b,c)\sim(a',b',c')\iff a=a'$. Note that $X^\#=X$ and $\pi(X)$ is the union of two fixed points $(+1)^\infty$, $ (-1)^\infty$ and a heteroclinic orbit: $\{\sigma^k((+1)^\infty0\cdot(-1)^\infty):k\in\ZZ\}$. Note also  that $X^\rec$ is the union of:
 \begin{itemize}
  \item[-] four $3$-periodic orbits mapped to the two fixed points, defining Bowen equivalence classes with $6$ elements each;
  \item[-] four heteroclinic orbits, each  mapped to the heteroclinic orbit, defining Bowen equivalence classes with $4$ elements each.
\end{itemize} 
The following is easily checked:
 $$\begin{aligned}
     \inf_{w\in\cL(X^\#)} \delta_{\sim}(w) = \delta_\sim((0,0,0),0) = 1 &< \inf\{|\cls{x}|:x\in X^\#\} = 4 \\
     &< \inf\{|\cls{x}|:x\in X^\rec\} = 6\\
     &\;= \inf_{w\in\cL(X^\rec),0\leq i<|w|} |\{v^i:v\in\cL(X^\rec),v\sim w\}\\
     &\;= \inf_{w\in\cL(X^\rec),0\leq i<|w|} \degree_\sim(w) =\delta_\sim((1,0,0),0).
  \end{aligned}$$
\end{example}
}

To prepare for the proof of Theorem~\ref{thm-degree}, we fix a magic couple $(W,I)$ in $X^\#$ over $X^\rec$. Since $\sim$ is locally finite, $M=|\cls{W}|$ is finite. We enumerate its elements and the symbols at index $I$:
 $$
   \cls{W}=\{W^1,\dots,W^M\} \text{ and }\{a_1,\dots,a_d\}=\{W^1_I,\dots,W^M_I\}
 $$
where $d=\degree_\rec(\sim)$. Obviously, $d\leq M$. We can assume: $a_i=W^i_I$ for $i=1,\dots,d$.

\newcommand\Trans{\mathcal T}
To any word that can be written as a concatenation $WuW$, we associate:
 $$
    \Trans_{ij}(WuW):=\{a_i\bar\tau: va_i\bar\tau a_j w\in\cL(X^\#),\; va_i\bar\tau a_j w\sim WuW\text{ for some }|v|=I,\; |w|=\Size{W}-I-1\}
 $$
for $1\leq i,j\leq d$.
We call the words $a_i\bar\tau\in \Trans_{ij}(WuW)$ \emph{transitions}. Note that these words have the same length as $Wu$.
Matching transitions can be concatenated:

\begin{claim}\label{claim-concat-trans}
For any $1\leq i,j,k\leq d$, any $WuWu'W\in\cL(X^\rec)$, there is an injection:
 $$
   \Trans_{ij}(WuW)\times\Trans_{jk}(Wu'W)\to\Trans_{ik}(WuWu'W),\quad 
   (\tau,\tau')\longmapsto \tau\tau'.
  $$
\end{claim}

\begin{proof}
Let $\tau,\tau'$ be as above. We can write $\tau=a_i\bar\tau$, $\tau'=a_j\bar\tau'$.
By definition, $va_i\bar\tau a_jw\sim WuW$ and $v'a_j\bar\tau'a_kw'\sim Wu'W$ where $v,w,v',w'$ are words of lengths  $|v|=|v'|=I$ and $|w'|=|w|=|W|-I-1$. Thus, the concatenation $va_i\bar\tau a_j\bar\tau'a_kw'$ belongs to $\cL(X^\#)$  and is related to $WuWu'W$. Hence $a_i\bar\tau a_j\bar\tau'\in\Trans_{ik}(WuWu'W)$. The injectivity is obvious.
\end{proof}

Any transition can be extended to the right and to the left:

\begin{claim}\label{claim-extend-trans}
For any $1\leq i\leq d$, any $WuW\in\cL(X^\rec)$, the following two sets are not empty:
 $$
    \Trans_{i*}(WuW):=\displaystyle\bigcup_{1\leq j\leq d} \Trans_{ij}(WuW)
     \text{ and }     \Trans_{*i}(WuW):=\displaystyle\bigcup_{1\leq j\leq d} \Trans_{ji}(WuW).
 $$
Moreover, $\Trans_*(WuW):=\bigcup_{1\leq i\leq d} \Trans_{i*}(WuW)=\bigcup_{1\leq j\leq d} \Trans_{*j}(WuW)$ has at least $d$ elements.
\end{claim}

\begin{proof}
Obviously,
 $
    \{m_I:m\in\cL(X^\#),\, m\sim WuW\}\subset\{w_I:w\in\cL(X^\#),\, w\sim W\}=\{a_1,\dots,a_d\}.
 $
Since $(W,I)$ is magic, the cardinalities are equal and finite so the inclusion is an equality. Since $\{m_I:m\sim WuW\}=\{a:a\tau\in\Trans_*(WuW)\}$, it follows that $\Size{\Trans_*(WuW)}\geq d$. It also follows that $\Trans_*(WuW)$ contains a word beginning with $a_i$ so that $ \Trans_{i*}(WuW)$ is not empty. Likewise $\Trans_{*i}(WuW)$ is not empty.
\end{proof}

 A word $u$ will be called \emph{special} if $\Trans_{i*}(WuW)$ has more than one element for some $1\leq i\leq d$.

\begin{lemma}\label{prop-magic}
Fix a magic word $W$.  
Assume that $x\in X^\rec$ is an infinite concatenation
 \begin{equation}\label{eq-groups}
    \dots W u^{-1} W u^0 W u^1 W \dots
 \end{equation}
 where each $u^k$ is some word.
More precisely, there is an increasing integer sequence  $(n_k)_{k\in\ZZ}$ such that, for all $k\in\ZZ$, $x_{[n_k,n_{k+1}-1]}=Wu^k$.

If there are (at least) $0\leq K\leq\infty$ distinct integers $k\in\ZZ$ such that $u^k$ is special,
then
 $$
    \Size{\cls{x}}\geq K+d.
 $$
Moreover, for each $k\in\ZZ$, $\{y_{n_k+I}:y\in\cls{x}\}=\{a_i:i=1,\dots,d\}$. 
\end{lemma}

\begin{proof}
We may and do assume that $\cl{x}$ is finite and that $K$ is finite (by an easy reduction). To simplify notation, we assume that $n_0=-I$ and that there are $K$ positive integers $k$ with $u^k$ special (using shift invariance). For each $n\geq0$, let $K(n)$ be the number of integers $0<k<n$ with $u^k$ special. For each $n\geq1$, let $U^n:=u^0W\dots W u^{n-1}$. 

We claim that for every $n\geq1$,
 \begin{equation}\label{eq-ext-growth}\begin{aligned}
   &\Trans_{*}(WU^{n}W):=\bigcup_{1\leq i\leq d}  \Trans_{i*}(WU^{n}W) 
   \text{ has at least $K(n)+d$ elements}\\
   &\{w_0:w\in\Trans_{*}(WU^nW)\}=\{a_1,\dots,a_d\}.
 \end{aligned}\end{equation}

We proceed by induction. Claim \ref{claim-extend-trans} implies that $\Size{\Trans_*(Wu^0W)}\geq d$ which is eq.~\eqref{eq-ext-growth} for $n=1$ since $K(1)=0$.  Assume eq.~\eqref{eq-ext-growth} for some $n\geq1$. Claims \ref{claim-concat-trans}~and~\ref{claim-extend-trans}, show that each element of $\Trans_*(WU^nW)$ can thus be extended to an element of $\Trans_*(WU^{n+1}W)$. Thus $|\Trans_*(WU^{n+1}W)|\geq |\Trans_*(WU^nW)|$ and $\{w_0:w\in\Trans_{*}(WU^{n+1}W)\}\supset\{w_0:w\in\Trans_{*}(WU^nW)\}$. Eq.~\eqref{eq-ext-growth} follows if $K(n+1)=K(n)$. Otherwise $K(n+1)=K(n)+1$ and $Wu^nW$ is special so some element of $\Trans_*(WU^nW)$ has at least two distinct extensions in $\Trans_*(WU^{n+1}W)$. This completes the induction and proves the claim \eqref{eq-ext-growth}.

\medbreak

Observe that the words in $\Trans_*(WU^nW)$ are the prefixes of length $|WU^n|$ of the words in $\Trans_*(WU^{n+1}W)$. Hence one can take an inductive limit and obtain $\mathcal Y\subset\cA^{[0,\infty)}$ such that, for each $n\geq0$, $\Trans_*(WU^nW)=\{y_{[0,|WU^n|-1]}:y\in\mathcal Y\}$. It is easy to see that ($\ZZ_-:=\{0,-1,-2,\dots\}$):
 \begin{itemize}
  \item[-] $\mathcal Y$ has at least $K+d$ elements;
  \item[-] each $y\in\mathcal Y$ satisfies: $y_0\in\{a_1,\dots,a_d\}$, $y_n\stackrel{X}\to y_{n+1}$ for all $n\in\NN$, and $y\sim x_{[0,\infty)}$. 
\end{itemize}
For each $1\leq i\leq d$, an analogous use of Claims \ref{claim-extend-trans} and \ref{claim-concat-trans} provides an infinite one-sided sequence $z^i\in\cA^{\ZZ_-}$ such that $z^i_0=a_i$ and $z^i_n\stackrel{X}\to z^i_{n+1}$ for all $n<0$, and $z^i\sim x_{(-\infty,0]}$. 
Therefore $\{z^{y_0}_{(-\infty,-1]}y_{[0,\infty)}):y\in\mathcal Y\}\subset\cl{x}$ so that $\Size{\cls {x}}\geq K+d$.
\end{proof}

\begin{proof}[Proof of Theorem \ref{thm-degree}]
Let $x\in X^\rec$ with finite class $\cls{x}$. First,  we assume that $x$ sees no magic word i.o. Since $x\in X^\rec$, no magic word can appear in $x$. By Lemma~\ref{lem-deg-less-fiber}, $|\cls{x}|\geq\delta_\sim(x)>\delta_\rec(\sim)$.

Conversely, we assume by contradiction that  $x$ sees i.o. some magic word $W$ and that $\Size{\cls{x}}\ne\degree_\rec(\sim)$. By Lemma \ref{lem-deg-less-fiber}, this implies that $\Size{\cls{x}}\geq\degree_\rec(\sim)+1$.
We decompose $x$ as in eq.~\eqref{eq-groups} (remark that the magic word $W$ can occur inside the fillers $u^k$). 
For $k$ large enough:
 $$
   \Size{\{y_{[n_{-k}+I,n_k-|W|+I]}:y\in\cls{x}\}}\geq \degree_\rec(\sim)+1.
 $$
Thus one can find $y,y'\in\cls{x}$ such that
 $$
   y_{n_{-k}+I}=y'_{n_{-k}+I} \text{ but }y_{[n_{-k}+I,n_k-|W|+I]}\ne y'_{[n_{-k}+I,n_k-|W|+I]}.
 $$
Hence, writing $a_i$ for $y_{n_{-k}+I}=y'_{n_{-k}+I}$,
 $$
   |\Trans_{i*}(Wu^{-k}W\dots W u^kW)|\geq2
 $$
that is, $u^{-k}W\dots W u^k$ is a special word. This contradicts the following claim  and therefore proves the theorem.
\end{proof}

\begin{claim}
No special word occurs in $x$.
\end{claim}

\begin{proof}[Proof of the claim]
Assume by contradiction that there is a special word $u^*$ such that $Wu^*W$ occurs in $x$. Since $x$ is recurrent, $Wu^*W$ occurs infinitely often. Select  a decomposition as in eq.~\eqref{eq-groups} such that $Wu^kW=Wu^*W$ for infinitely many integers $k$. By Lemma~\ref{prop-magic}, $\cls{x}$ must be infinite, a contradiction.
\end{proof}

\section{Injective codings}\label{s-main}

We use the previous constructions and results to build injective codings on larger and larger sets. We will first see that a Bowen quotient (Def.~\ref{def-Bowen-quotient}) may produce a coding with a large injectivity set. We will then see how to repeat this construction to capture all the image through suitable recodings.

\medbreak

For convenience, we recall some definitions. An \emph{excellent}   semiconjugacy (Def.~\ref{def-excellent}) is a Borel, finite-to-one semiconjugacy which admits a locally finite Bowen relation. Sometimes we will abuse notation denoting the Bowen relation  by the corresponding semiconjugacy (even though the semiconjugacy does not determine the Bowen relation).

The \emph{degree spectrum} of a semiconjugacy $\pi:X^\#\to Y$ is (Def.~\ref{def-deg-spectrum}):
 $$
    \Delta(\pi):=\{ n\geq1 : \{y\in Y: |\pi^{-1}(y)|=n\} \text{ is not a null set}\}.
 $$

\newcommand\magic{{\rm magic}}

\subsection{A Bowen quotient and its injectivity set}\label{s-injectivity}
We analyze the Bowen quotient construction using the magic word theory from the previous section. Let
 $$
      X_N^\magic:=\{x\in X_N: \exists w\in\cL(X_N^\rec)\; \text{such that $w$ is a magic word for $\pi_N$ and $x$ sees $w$ i.o.}\}.
  $$
In this subsection, we say that a function $C:\cA\times\cA\to\NN$ is a \emph{multiplicity bound} for $\pi:X^\#\to Y$ if for all $x\in X^\#$, $|\pi^{-1}(\pi(x))|\leq C(a,b)$ for every $(a,b)\in\cA\times\cA$ such that $x_{-n}=a$, resp. $x_n=b$, for infinitely many positive integers $n$.

We have  the following.

\newcommand\inj{{\operatorname{inj}}}
\begin{lemma}\label{lem-good-quotient}
Let $X$ be a Markov shift  and let $\pi:X^\#\to Y$ be an excellent semiconjugacy for some Bowen relation $\sim$. Let $N$ belong to its degree spectrum $\Delta(\pi)$ and let $(\pi_N:X_N^\#\to Y,\simN)$ be the Bowen quotient of $(\pi,\sim)$ with order $N$. The following holds:
 \begin{enumerate}
  \item $(\pi_N,\simN)$ is excellent and has degree  $\degree_\rec(\pi_N)=\min\Delta(\pi_N)=1$. Moreover, if $\pi$ admits a multiplicity bound, so does $\pi_N$;
 \item if $X$ is locally compact, so is $X_N$;
  \item $\pi_N=\pi\circ q_N$ with $q_N:X_N\to X$ a $1$-Lipschitz map such that $q_N(X_N^\#)\subset X^\#$;
  \item $q_N:X_N^\#\to X^\#$ is proper, i.e., $q_N^{-1}(K)\cap X_N^\#$ is compact for any  compact $K\subset X^\#$;
 \item $\pi_N(X_N^\#)\subset \{y\in\pi(X^\#): |\pi^{-1}(y)|\geq N\}$ and the difference is a null set;
 \item $\Delta(\pi_N) = \{ \binom{r}{N} : r\in\Delta(\pi)~,\; r\geq N\}$ and, for $r\in\Delta(\pi)$ with $r\geq N$, the set $\left\{y\in Y:|\pi_N^{-1}(y)|={\textstyle\binom{r}{N}}\right\}$:
  \begin{enumerate}[(i)]
   \item is included in  $\left\{y\in Y: |\pi^{-1}(y)|\geq r\right\}$;
   \item is equal to $\left\{y\in Y: |\pi^{-1}(y)|= r\right\}$ up to a null set;
  \end{enumerate}
 \item  $X_N^\magic\subset\{x\in X_N^\#:|\pi_N{}^{-1}(\pi_N(x))|=1\}$ and the difference is a null set.
 \end{enumerate}
\end{lemma}

\begin{proof} This is Theorem~\ref{thm-Bowen-quotient} and Corollary~\ref{c-degspec-N} except for the following points.

The computation of the degree $\degree_\rec(\pi_N)=\min \Delta(\pi_N)=\binom{N}{N}=1$ follows from  Corollary~\ref{cor-Xmin}. If there is a multiplicity bound $C$ for $\pi$, the following gives a multiplicity bound for $\pi_N$:
 $$
   |\pi_N^{-1}(y)|\leq \binom{|\pi^{-1}(y)|}{N}\leq C_N(A,B):=\sup_{(a,b)\in A\times B} C(a,b)^N/N!,
 $$
completing the proof of item (1). 

To prove item (7), note first that the inclusion follows from  Lemma~\ref{lem-magic-injec} and that, by Theorem~\ref{thm-degree}, the difference is included in $X_N\setminus X_N^\rec$  which is a null set.
\end{proof}

We deduce the following theorems for use in \cite{BCS}.

\begin{theorem}\label{thm-ae-injectivity}
Let $(X,S)$ be a Markov shift, $(Y,T)$ be a dynamical system, and let $\pi:X^\#\to Y$ be an excellent semiconjugacy. Let $\mu\in\Proberg(T)$ with $\mu(\pi(X^\#))=1$.
Then there exist a Markov shift $\hat X$ and a semiconjugacy $\hat\pi:\hat X^\#\to Y$ such that:
 \begin{enumerate}[\quad(1)]
  \item  $\hat\pi:\hat X^\#\to Y$ is an excellent semiconjugacy. Moreover, if $\pi$ admits a multiplicity bound, so does $\hat\pi$;
  \item if $X$ is locally compact, so is $\hat X$;
  \item $\hat\pi=\pi\circ q|_{\hat X^\#}$ where $q:\hat X\to X$ is a $1$-Lipschitz map with $q(\hat X^\#)\subset X^\#$;
  \item $q:\hat X^\#\to X^\#$ is proper, i.e., $q^{-1}(K)\cap\hat X^\#$ is compact for any  compact $K\subset X^\#$;
  \item $\hat\pi(\hat X^\#)\subset \pi(X^\#)$ and for $\mu$-a.e. $y\in Y$, $|\hat\pi^{-1}(y)|=1$;
  \item there is an invariant measure $\hat\mu$ on $\hat X$ such that $\hat\pi:(S,\hat\mu)\to(T,\mu)$ is an isomorphism;
  \item $X$ is irreducible.
  \end{enumerate} 
\end{theorem}

\begin{proof}
Observe that $y\mapsto|\pi^{-1}(y)|$ is a $T$-invariant function. By ergodicity, it has a $\mu$-a.e. constant and positive  value we denote $N$.  Obviously $N\in\Delta(\pi)$. Let $(\pi_N:X_N^\#\to Y,\simN)$ be the Bowen quotient of $(\pi,\sim)$ of order $N$ as in Lemma~\ref{lem-good-quotient}. Thus  $\pi_N$ satisfies all the claims above except possibly for items (5)-(7). 

Item~(6)(ii) for $r=N$  of the lemma implies that $|\pi_N^{-1}(y)|=1$ for $\mu$-a.e. $y\in Y$. Therefore, there is a unique $\hat\mu\in\Prob(S_{|X_N})$ such that $\pi_N:(\hat\mu,S_N)\to(\mu,T)$ is an isomorphism. Since $q_N(X_N^\#)\subset X^\#$, we have $\pi_N(X_N^\#)\subset\pi(X^\#)$. These remarks yield items (5)~and~(6). As $\bar\mu$ is ergodic, it is carried by an irreducible component $\hat X$ of $X_N$.  It is now clear that  $q:=q_N|\hat X$ and $\hat\pi:=\pi\circ q$ have all the claimed properties.

\end{proof}

\begin{theorem}\label{thm-large-injectivityset}
Let $X$ be a Markov shift  and let $\pi:X^\#\to Y$ be an excellent semiconjugacy.
Then there exist another Markov shift $\hat X$ and a semiconjugacy $\hat\pi:\hat X^\#\to Y$  such that properties (1)-(4) in Theorem~\ref{thm-ae-injectivity} hold and, moreover:
 \begin{enumerate}
  \item[\quad(5')] $\hat\pi(\hat X^\#)\subset \pi(X^\#)$ and the difference is a null set;
  \item[\quad(6')] there is a word $\hat W\in\cL(\hat X^\rec)$ s.t. for any $x\in\hat X$ that sees i.o. $\hat W$, $\hat\pi^{-1}(\hat\pi(x))=\{x\}$;
  \item[\quad(7')] if $X$ is irreducible, then so is $\hat X$.
 \end{enumerate} 
\end{theorem}

\begin{remark}
As noted in the proof below, our argument gives a stronger result than stated in item~(5'). If $X$ is irreducible, we obtain $\pih(\hat X^\#)=\pi(X^\#)$. 
In the general case, the spectral decomposition of $X$ into  irreducible components $X_i$, ${i\in I}$, shows that:
 $
   \bigcup_{i\in I} \pi(X_i^\#)\subset\hat\pi(\hat X^\#)\subset\pi(X^\#).
 $
\end{remark}

\begin{proof}[Proof of Theorem~\ref{thm-large-injectivityset}]
Let $\hat\pi:\hat X^\#\to Y$ with $\hat\sim$ be the Bowen quotient of order $N=\min\Delta(\pi)$.  Lemma~\ref{lem-good-quotient} yields items (1)-(4) and (5') and $\degree_\rec(\pih)=1$. 
Theorem~\ref{thm-degree} implies item (6') for any  magic word $\hat W$  for $\hat\sim$.

\smallbreak
We now assume that $X$ is irreducible and prove that item (7') can be satisfied while keeping the other properties. We first observe that items (1)-(4) and the inclusion in item (5) are obviously preserved by restriction to any irreducible component. We are going to select an irreducible component for which the second half of item (5') and item (6') are satisfied.

\medbreak

During this proof, we will say that a sequence $\hat x\in\hat X^\#$ is \emph{related} to some $X$-word $w$, if there are infinitely many $p\geq0$ and infinitely many $p\leq 0$ such that: 
 $$
    \forall 0\leq i<|w| \quad w_{i}\in\hat x_{p+i}.
 $$ 

\medbreak\noindent
{\it Observation.} For any periodic $x\in X$, $|\pi^{-1}(\pi(x))|\geq N:=\min\Delta(\pi)$ and therefore by Fact~\ref{fact-special-lift}, there exists $\hat x\in\hat X^\#$ with $x_n\in\hat x_n$ for all $n\in\ZZ$.

\begin{claim}
There is an irreducible component $\hat Z$ of $\hat X$ that contains any $\hat x\in X^\#$ related to some magic word for $\sim$ over $X^\rec$. Moreover, $\degree_\rec(\hat\pi_{|\hat Z^\#})=1$.
\end{claim}

\begin{proof}[Proof of the claim]
Let $(w,i)$ be a magic couple  for $\pi_{|X^\#}$ over $X^\rec$. Define the set $A_w:=\{v_i:v\in\cL(X^\#)$ s.t. $v\sim w\}$ with cardinality $|A_w|=\degree_\rec(\pi)$. Note that since $w\in\cL(X^\rec)$, there is a  periodic point $x\in X^\#$ that sees $w$. By the observation, this implies the existence of $\hat x\in\hat X^\#$ related to $w$.

Now let $\hat x\in\hat X^\#$ be related to $w$.  By Fact~\ref{fact-lift-quotient}, for all $n\in\ZZ$ $\hat x_{n}=\{z^a_{n}:a\in\hat x_0\}$ where for each $a\in\hat x_0$, $z^a=Q(\hat x,x)$. Hence, for all $a\in\hat x_0$, $z^a_p\dots z^a_{p+|w|-1}\sim w$ and  $\hat x_{p+i}\subset A_w$. Since these sets have equal cardinalities, we have: $\hat x_{p+i}=A_w$. 
Therefore, $\hat x$ belongs to the irreducible component $\hat Z_w$ of $\hat X$ containing the symbol $A_w$.

If $v$ is another magic word,  there is a periodic orbit $x\in X$ that sees i.o. $v$ and also sees i.o. $w$ ($X$ is transitive). The observation yields some $\hat x\in\hat X^\#$ which is related to both $v$ and $w$ so 
 $\hat Z_v=\hat Z_w$. Thus there is an irreducible component $\hat Z$ that contains all $\hat x\in\hat X^\#$ related to any magic word for $\pi$.

To show that $\degree_\rec(\pih|_{\hat Z^\#})=\degree_\rec(\pih)=1$, it suffices to find a magic word for $\pih$ in $\cL(\hat Z^\rec)$. Let $w$ be a magic word for $\pi$. Given a periodic $x\in X^\#$ with $w$ occuring at index $0$, the observation yields a periodic, hence word recurrent $\hat x\in\hat X^\#$ with $w_n\in\hat x_n$ for all $0\leq n<|w|$. 

Let $\hat w:=\hat x_0\dots \hat x_{|w|-1}$. Obviously $\hat w\in\cL(\hat Z^\rec)$. We check that $\hat w$ is a magic word for $\pih$.
Let  $\hat v\in\cL(\hat X^\rec)$ such that $\hat v\;\hat\sim\;\hat w$. By Fact~\ref{fact-lift-quotient}, $\hat v_n=\{v^a_n:a\in \hat v_0\}$ where $v^a\in\cL(X^\#)$   for all $0\leq n<|\hat w|$. In particular $v^a_n\sim w_n$ for all $0\leq n<|w|$. As above, it follows that $\hat v_i=\{v^a_i:a\in\hat v_0\}=A_w$. This implies that $\degree_\rec(\pih|_{\hat Z^\#})=\delta_{\hat\sim}(\hat w,i)=1$.
\end{proof}

Let $x\in X^\#$. We are going to show that $x\in\pih(\hat Z^\#)$ by finding $y\in X^\#$ Bowen equivalent to $x$ and which can be approximated by  $q(\hat x^n)$ with periodic $\hat x^n\in\hat Z^\#$. Fix a magic word $w$ for $\pi_{|X^\#}$. There are symbols $a,b$ of $X$ and integers  $m_k,n_k\geq k$ such that $x_{-m_k}=a$ and $x_{n_k}=b$ for all $k\geq1$. There is an $X$-word $u_0\dots u_{\ell+1}$, $\ell\geq1$, with $u_0=b$ and $u_{\ell+1}=a$ and containing $w$ as a subword (since $X$ is irreducible). For each $k\geq1$, let $x^k\in X^\#$ be the periodic sequence with  period $\tau_k:=n_k+m_k+\ell+1$ defined by:
  $$
    \forall i=-m_k,\dots,n_k+\ell\quad 
      x^k_{i} = \alter{ x_i &\text{ if }-m_k\leq i\leq n_k,\\
                                u_{i-n_k+1} & \text{ if }n_k\leq i<n_k+\ell.}
  $$
Note that for all $i\in\ZZ$, $\cA_i:=\{x^k_i:k\geq1\}\subset\{x_i:k\leq |i|\}$ is finite. The local finiteness of the Bowen relation implies that, for all $i\in\ZZ$, the set of symbols $\cB_i:=\{s:\exists t\in\cA_i$ s.t. $s\sim t\}$ is finite. 

Since $x^k\in X^\#$ is periodic, the observation gives  $\hat x^k\in\hat X^\#$ such that  $x^k_i\in\hat x^k_i$ for all $i\in\ZZ$.
In particular, $\hat x^k$ is related to $w$ so it belongs to $\hat Z$ by the claim. Note that  $\hat x^k_i\subset\{s:s\sim x^k_i\}$ hence $\hat x^k_i\subset\cB_i$ for all $i\in\ZZ$ and $k\geq1$. Thus  there is a point of accumulation $\hat x=\lim_n \hat x^{k(n)}\in\hat Z$ for some sequence $k(n)\uparrow\infty$. If $x_i=a$ (resp. $b$), then, for all large $k$, $\hat x^k_i\subset\{c:c\sim a$ (resp. $c\sim b$)$\}$ which is finite and independent of $i\in\ZZ$, hence $\hat x\in \hat Z^\#$. 

Let $y^k:=q(\hat x^k)$ for $k\geq1$. As $q$ is continuous, $y^{k(n)}=q(\hat x^{k(n)})$ converges to the sequence $y:=q(\hat x)$. Since $\hat x\in\hat X^\#$, we have $y\in X^\#$. For all $i\in\ZZ$, $y^k_i\in\hat x^k_i$ by construction, hence $y^k_i\sim x^k_i$.  Recalling that $x^k_i=x_i$ for all $k\geq|i|$, we get $y\BE x$. By the Bowen property:
 $$
  \pi(x)=\pi(y)=\pi(q(\hat x))=\pih(\hat x).
 $$
 Thus $\pi(X^\#)\subset\pih(\hat Z^\#)=\pi(q(\hat Z^\#))\subset\pi(X^\#)$, so $\pih(\hat Z^\#)=\pi(X^\#)$, yielding item (5'). Since $\deg_\rec(\pi_{|\hat Z^\#})=1$,  Theorem~\ref{thm-degree} yields item (6'). 
\end{proof}

\begin{remark}
The periodic approximation argument in the last part of the proof of Theorem~\ref{thm-large-injectivityset} is partly inspired by some geometric construction of \cite{BCS}.
\end{remark}

\bigbreak

\subsection{Preparations}
We turn to the proof of the Main Theorem. Let $\pi:X^\#\to Y$ be an excellent  semiconjugacy for some Bowen relation $\sim$. We are going to build an injective coding of the image $\pi(X^\#)$.
We start with the following simple fact about partially ordered sets. In this paper $\NN$ is the set of nonnegative integers.

\begin{fact}\label{fact-reorder}
Let $(\cO,\preceq)$ be a countable (possibly finite) set together with a partial order $\preceq$. There is a bijection $\sigma:\{n\in\NN:n<|\cO|\}\to\cO$ which is nondecreasing, i.e., 
 \begin{equation}\label{eq-orders}
   \forall i,j: \quad \sigma(i)\preceq\sigma(j)\implies i\leq j
 \end{equation}
if and only if all initial segments  $\{b\in\cO:b\preceq a\}$, $a\in\cO$, are finite. 
\end{fact}

\begin{proof}
If $\sigma:\NN\to\cO$ is a bijection satisfying eq.~\eqref{eq-orders}, then any initial segment $\{b\in\cO:b\preceq\sigma(i)\}$ is finite as a subset of $\sigma(\{0,1,\dots,i\})$. We now assume that all initial segments are finite and proceed to build the bijection $\sigma$.

If $\cO$ is finite, then one can define  $\sigma:\{0,\dots,n-1\}\to\cO$ inductively by choosing, for each $0\leq k<n$,  $\sigma(k)$ to be some minimal element $\alpha$ among $\cO\setminus\sigma(\{0,\dots,k-1\})$, i.e., such that:
  $$
    \forall \beta\in\cO\setminus\sigma(\{0,\dots,k-1\})\quad
        \beta\preceq \alpha\implies\beta=\alpha.
    $$
We assume now that $\cO$ is infinite so there is a bijection $s:\NN\to\cO$. We define integers $N_0<N_1<\dots$ and $\sigma|\{0,\dots,N_n-1\}$ inductively by setting $N_0=0$ and, for each $n\geq0$,
 \begin{enumerate}
  \item applying the finite case, enumerate  $\{b\in\cO:b\preceq s(n)\}\setminus\sigma(\{0,\dots,N_n-1\})$ as $\{b_{n,1},\dots,b_{n,\ell_n}\}$ where $i\mapsto b_{n,i}$ is injective and non-decreasing;
  \item set $N_{n+1}:=N_n+\ell_n$ and $\sigma(N_n+i)=b_i$ for $i=0,\dots,\ell_n-1$.
 \end{enumerate}
It is easy to check that $\sigma$ is a nondecreasing bijection.
\end{proof}

\newcommand\cW{{\mathcal W}}

\medbreak

We will apply the following elementary construction to an enumeration of the magic words for the relation $\sim$ in $X^\#$ over $X^\rec$.

\begin{lemma}\label{lem-magic-subset}
Let $X$ be a Markov shift. Let  $\cW:=(W^j)_{1\leq j<J}$ ($1<J\leq\infty$) be an enumeration of $X$-words.  Then there is an injective one-block code $p:S\to X$ defined on a Markov shift $S$  whose image $p(S)$ is $X_\cW\setminus N$ for some null set $N$ and
 $$
   X_\cW:=\{x\in X:\text{there is $w\in\cW$ such that $x$ sees i.o. $w$}\}\subset X^\#.
 $$
\end{lemma}

Recall that, given some word $W\in\cL(X)$, $X_W$ is the set of sequences that see i.o. $W$ (see p.~\pageref{def-sees-io}).

\begin{proof}
Since the set of subwords of a given word is finite, Fact \ref{fact-reorder} allows us to assume (maybe after a permutation) that:
 \begin{equation}\label{eq-subord}
  W^i \text{ subword of }W^j \implies i\leq j
 \end{equation}

For $ 1\leq j< J$, consider the following subset of $X_\cW$:
 $$
   X_j:=\{x\in X: x\text{ sees i.o.  $W^j$, none of $W^1,\dots,W^{j-1}$ occurs in }x\}.
 $$
The injective code we are going to build will have image $\bigcup_{1\leq j<J} X_j$.
The sets $X_j$ are pairwise disjoint. Each sequence in  $X_\cW\setminus \bigcup_{1\leq j<J} X_j$  contains some $\cW$-word that does not occur i.o., hence this difference is null.

Fix $1\leq i<J$ such that $X_i\ne\emptyset$. Let $N:=|W^i|$. We perform some standard graph constructions. First, consider the $N$th higher block presentation $X^{[N]}$ of $X$ (see, e.g., \cite[I.4.1]{LindMarcus-book}) defined by the graph $\cG^1$ with:
 $$
    \text{vertices: }(x_0,\dots,x_{N-1}),\quad
    \text{arrows: } (x_0,\dots,x_{N-1})\to(x_1,\dots,x_N) \qquad (x\in X).
 $$
There is a topological conjugacy $X^{[N]}\to X$ defined by the one-block code
 $
    (x_0\dots x_{N-1}) \mapsto x_0
 $
with inverse: $(x_n)_{n\in\ZZ}\mapsto (x_{[n,n+N-1]})_{n\in\ZZ}$.

Let $\cG^2$  be the loop graph at the base vertex $W^i$ in $\cG^1$, defined as follows (see, e.g., \cite{GurevicSavchenko-1998,BoyleBuzzi-2017}). The \emph{first return loops} at $W^i$ are the finite sequences $(y_0,y_1,\dots,y_{k-1})$ where each $y_i$ is an $X$-word of length $N$,  $y_0=W^i$,  $y_0\to y_1\to\dots\to y_{k-1}\to W^i$ on $\cG^1$, and  $y_\ell\ne W^i$ for any $1\leq \ell\leq k-1$. Now the loop graph $\cG^2$ is defined by taking as vertices the couples $(v,\ell)$ where $v$ is a first return loop at $W^i$ and $0\leq \ell<|v|$, and as arrows:\footnote{Contrary to usual practice, we do not identify all $(v,0)$ vertices with a single distinguished vertex.}
 $$\begin{aligned}
    &(v,\ell)\to (v,\ell+1) \text{ if $v$ is a first return loop and }1\leq \ell+1<|v|\\
    &(v,|v|-1)\to(w,0) \text{ if $v,w$ are first return loops}.
 \end{aligned}$$
The corresponding shift is mapped into $X$ by the one-block code $(v,\ell)\mapsto v^\ell_0$ (i.e., the first symbol of the word $v^\ell$).

 We define $\cG^3$ by keeping from $\cG^2$ only the vertices $(v,\ell)$ where $v$ is a first return loop $(y_0,\dots,y_{k-1})$ that is \emph{good}, i.e., whose \emph{extensions}:
  $$
     (y_0,\dots,y_{k-1},W^i)
  $$
  map to an $X$-word of length $k+|W^i|$ that does not contain any of the words  $W^1,\dots,W^{i-1}$.
  
Let $S_i$ be the Markov shift defined by this loop graph $\cG^3$ and define $p^i:S_i\to X_i$ to be the restriction of the previous map.

\begin{claim} The map $p^i:S_i\to X_i$ is a topological conjugacy defined by a one-block code.
\end{claim}

\begin{proof}[Proof of the claim]
It is obvious that $p^i$ is a one-block code. We have to check that it defines a bijection and that its inverse is continuous.

Consider some $x\in X_i$.  It can be lifted to a concatenation of  first return loops since $x$ sees i.o. $W^i$ . These first return loops must be good since $x$ avoids $W^1,\dots,W^{i-1}$. Thus $x$ belongs to the image of $\cG^3$. 
Conversely, let $x\in X$ be the image of some $\hat x$ on $\cG^3$, i.e., an infinite concatenation of good first return loops.  Assume by contradiction that  some $W^j$, $j<i$ occurs in $x$. 
By \eqref{eq-subord}, this occurrence may overlap but cannot contain any occurrence of $W^i$.  Thus $W^j$ occurs in the image of some extended first return loop, so the first return loop is not good. This contradicts the definition of $\cG^3$, proving that $p^i(S_i)=X_i$.

Note that $p^i$ is invertible with inverse defined by:
 $$
    \forall x\in X_i\quad p^i(\hat x)=x\!\!\implies\!\! \hat x_0=((x_{[j,j+N-1]})_{-n\leq j< m}),-n)
 $$
with $n=\max\{k\leq0: x_{[k,k+N-1]}=W^i\}$ and $m=\min\{k>0:x_{[k,k+N-1]}=W^i\}$. This inverse is continuous. The claim is proved.
\end{proof}

To conclude the proof of the lemma, let $S$ be the Markov shift $\bigsqcup_{1\leq j<J} S^j$ (considering the alphabets to be pairwise disjoint) and define the map $p:S\to \bigcup_{1\leq i<J} X_i$ by $p(x)=p^{i}(x)$ if $x\in S^i$. This is well-defined. Obviously $p$ is a one-block code. As the sets $X^i$ are disjoint and each $p^i$ is injective, so is $p$. Remark that $p(S)=\bigcup_{1\leq j<J} X_j$ and that  this union coincides with $X_\cW$ up to a null set. 
\end{proof}

\begin{remark}
The proof of the above lemma does not provide a locally compact Markov shift $S$, even if $X$ is compact.
\end{remark}

\subsection{Proof of the Main Theorem}
Let $\piS:X^\#\to Y$ be an excellent semiconjugacy with a Bowen relation $\sim$.
We are going to divide the image $\pi(X^\#)$ according to the number of preimages and then successively reduce each of these numbers to one (ignoring null sets). We assume that $\pi(X^\#)$ is not null as otherwise there is nothing to show.

\medbreak

Let $(\Delta(i))_{1\leq i<I}$ with $1<I\leq\infty$ be the increasing enumeration of the degree spectrum $\Delta(\pi)$ (see Def.~\ref{def-deg-spectrum}).
The corresponding \emph{degree partition} of $\pi(X^\#)$ is:
 $$
   (Y^i)_{1\leq i<I} \text{ with } Y^i:=\{y\in Y : \Size{\pi^{-1}(y)}=\Delta(i)\}.
 $$

We are going to define  semiconjugacies $\pi_i:Z_i^\#\to Y$ with Bowen relations~$\stackrel{i}\sim$ such that, setting $\widetilde Z_i:=\{x\in Z_i^\#:x$ sees i.o. some magic word for $\pi_i\}$ we have:
 \begin{enumerate}[\qquad(a)]
  \item $\pi_i:Z_i^\#\to Y$ is an excellent semiconjugacy with $\degree_\rec(\pi_i)=1$;
  \item $\pi_i=\pi\circ q_i|_{Z_i^\#}$ for some $1$-Lipschitz map $q_i:Z_i\to X$;
  \item for any $y\in\pi_i(\widetilde Z_i)$, $|\pi_i^{-1}(y)|=1$;
  \item $\pi_i(\widetilde Z_i)=Y_i$ up to a null set;
  \item the degree partition of $\pi_i$ is $(Y^i_j)_{i\leq j<I}$ with $Y^i_j:=\{y\in Y:|\pi_i^{-1}(y)\cap Z_i^\#|=\Delta_i(j)\}$ equal to $Y_j$ up to a null set and $(\Delta_i(j))_{i\leq j<I}$ is the increasing enumeration of the degree spectrum;
 \end{enumerate}
 
For $i=1$, we let $(\pi_1:Z_1^\#\to Y,\stackrel{1}\sim)$ be the Bowen quotient of the semiconjugacy $(\piS,\sim)$ with order $\Delta(1)=\min\Delta(\pi)$. Since $\pi:X^\#\to Y$ is an excellent semiconjugacy, Lemma~\ref{lem-good-quotient} shows that this is well-defined  and that the above items (a)-(e) hold with $\Delta_1(j)=\binom{\Delta(j)}{\Delta(1)}$.

Let $1<i<I$ and assume that  $(Z_{j},\piS_{j},\stackrel{j}\sim,q_{j})$ have been defined with these properties  for all $1\leq j<i$. Let $(\pi_i:Z_i^\#\to Y,\stackrel{i}\sim)$ be the Bowen quotient of $(\pi_{i-1},\stackrel{i-1}\sim)$ of order $N_{i-1}:=\Delta_{i-1}(i)$ (this last set is nonempty by item (e) since $i<I$). Lemma~\ref{lem-good-quotient} shows that this is well-defined  and that the above items (a)-(d) hold with $q_i=q_{i-1}\circ q$ where $q$ is given by Lemma~\ref{lem-good-quotient}, item~(3).

We turn to item (e). The item (6) of Lemma~\ref{lem-good-quotient} shows that $\Delta(\pi_i)=\{\Delta_i(j):i\leq j<I\}$ with $\Delta_i(j):=\binom{\Delta_{i-1}(j)}{\Delta_{i-1}(i)}$ and, for all $i\leq j<I$, up to a null set:
 $$
    \{y\in Y:|\pi_i^{-1}(y)\cap Z_i^\#|=\Delta_i(j)\} =  
     \{y\in Y:|\pi_{i-1}^{-1}(y)\cap Z_{i-1}^\#|=\Delta_{i-1}(j)\} 
   $$
 so $Y^i_j=Y^{i-1}_j=Y_j$ up to null sets, proving~(e).

\begin{claim}\label{claim-recode}
For any $1\leq i<I$, 
there are a Markov shift $S_{i}$ and a one-block code $p_{i}:S_{i}\to p_i(S_i)$ such that: $p_i(S_i)\subset \widetilde Z_i$ with the difference a null set, $\pi_i\circ p_i(S_i)=Y_i$ up to a null set, and $\pi_{i}\circ p_{i}$ is injective.
\end{claim}

\begin{proof}[Proof of the claim]
Lemma~\ref{lem-magic-subset} provides  an injective one-block code $p_i$ of some Markov shift $S_i$ into $\tilde Z_i$ with  $p_i(S_i)\subset\widetilde Z_i$ with the difference a null set. 

By item (d) above and Lemma~\ref{lem-null-sets}, $\pi_i\circ p_i(S_i)=Y_i$ up to a null set.
By item (c), $\pi_i\circ p_i$ is injective.
\end{proof}

To conclude, let $S$ be the disjoint union $\bigsqcup_{1\leq i<J} S^i$ of the one-block codes from the previous claim (we can always recode to ensure this disjointness). It is a Markov shift. Define a semiconjugacy $p$ on $S$  by:
 $$
   p|S^i=\pi_i\circ p_i.
 $$ 

Note that $p|S^i=\pi\circ q_i\circ p_i$ where $p_i$ is a one-block code and $q_i$ is $1$-Lipschitz. Thus $q_i\circ p_i$ is $1$-Lipschitz. The image of $p$ contains $\bigcup_{1\leq i<J} Y^i$ up to a null set, hence $\pi(X^\#)$ up to a null set. To conclude the proof of the Main Theorem, it suffices to see that the images $p(S_i)\subset \pi_i(\widetilde Z_i)$, $1\leq i<I$,  are pairwise disjoint. We have:

\begin{claim}\label{claim-disjoint}
For any $1\leq j<i<I$, $\pi_i(Z_i^\#)\cap\pi_j(\widetilde Z_j)=\emptyset$. In particular, the images $\pi_i(\widetilde Z_i)$, $1\leq i<I$, are pairwise disjoint.
\end{claim}

 To prove this claim, note that $y\in\pi_i(Z_i^\#)$ implies that $|\pi_{i-1}^{-1}(y)\cap Z_{i-1}^\#|\geq\Delta_{i-1}(i)>1$ and thus, by induction, $|\pi_{j}^{-1}(y)\cap Z_{j}^\#|\geq\Delta_{j}(i)>1$.
 However, $y\in\pi_j(\widetilde Z_j)$ implies $|\pi_j^{-1}(y)\cap Z_{i-1}^\#|=1$. Hence $\pi_i(Z_i^\#)\cap\pi_j(\widetilde Z_j^\#) = \emptyset$ as claimed. The last assertion follows from $\widetilde Z_i\subset Z_i^\#$.
 
\medbreak

\begin{remark}
The Bowen quotient is used for two seemingly distinct purposes: first, to remove points whose images have already been taken care of; second, to  lower the minimal degree to~$1$.
\end{remark}

\section{Applications to surface diffeomorphisms}\label{s-applications}

\newcommand\lip{\operatorname{Lip}}

We prove Theorem~\ref{theorem-SurfDiffCoding} and a more precise version of Theorem~\ref{theorem-SurfDiffPeriodic}. 

\medbreak
Let $f$ be a $C^{1+\alpha}$-diffeomorphism, $\alpha>0$, of a smooth closed surface $M$ with $h_\top(f)>0$. 
Let $\chi>0$ and $\tilde\chi<\chi$ (arbitrarily close to $\chi$, see below). Recall that a measure is $\chi$-hyperbolic if it has one positive exponent larger than $\chi$ and one negative exponent less than $-\chi$.

As observed in \cite[Sec. 8]{BoyleBuzzi-2017}, Sarig \cite{Sarig-JAMS} provides  a Markov shift $\hat\Sigma$ and a H\"older-continuous semiconjugacy $\hat\pi:\hat\Sigma\to M$ such that, $\hat\Sigma^\#$ denoting its regular part:
\begin{itemize}
 \item[(P1)] $\hat\pi|\hat\Sigma^\#$ admits a Bowen relation (called \emph{affiliation} in  \cite[Sec. 12.3]{Sarig-JAMS}) which is locally finite (see \cite[Summary 8.1(4)(5)]{BoyleBuzzi-2017});
 \item[(P2)] $\hat\pi|\hat\Sigma^\#$ is finite-to-one (as explained in \cite{Lima-Sarig} the claim that $\hat\pi$ is finite-to-one on $\Sigma$ itself was made erroneously in \cite{Sarig-JAMS});
 \item[(P3)] $\mu(\hat\piS(\hat\Sigma^\#))=1$ for any $\chi$-hyperbolic measure $\mu\in\Proberg(f)$.
 \item[(P4)] any ergodic $\nu$ on $\hat\Sigma$, $\hat\pi_*(\nu)$ is $\chi/2$-hyperbolic, see \cite[Prop. 12.6]{Sarig-JAMS}.
\end{itemize}
The construction and analysis of $\hat\Sigma$ in \cite{Sarig-JAMS} relies on another Markov shift $\Sigma$ defined by a graph $\mathcal G$ whose vertices $\Psi^{p^s,p^u}$ are  \emph{double  charts}, that is, local charts $\Psi:(-r,r)^2\to M$  centered at some point $x\in M$ together with two numbers $p^s,p^u>0$ with $r=\min(p^s,p^u)$.  
The charts $\Psi$ are defined by Pesin theory as $\exp_x\circ\, C_\chi(x)$ where $\exp_x$ is the exponential map centered at $x$ and $C_\chi(x)$ is the Oseledets-Pesin reduction matrix. These charts make ``the hyperbolicity of $f$ uniform'': for any arrow $\Psi^{p^s,p^u}\to\Phi^{q^s,q^u}$ in $\mathcal G$, the map $\Phi^{-1}\circ f\circ\Psi$ is close to a linear map $(x_1,x_2)\mapsto(\lambda x_1,\kappa x_2)$ with $\lambda>e^\chi$ and $\kappa<e^{-\chi}$. 

 Each Markov shift has its cylinders. For $\Sigma$, they are:
 $$
     Z_{-n}(\Psi_n,\dots,\Psi_n):=\pi\{x\in\Sigma:\forall |k|\leq n\; x_k=\Psi_k\}\subset M \text{ where each $\Psi_k$ is a double chart}
 $$
while those in $\hat\Sigma$ are:
 $$
     {}_{-n}[ R_{-n},\dots, R_n] := {\hat\pi}\{x\in\hat\Sigma:\forall |k|\leq n\; x_k=R_k\}\subset M \text{ where each $R_k$ is a rectangle}.
 $$

In this way,  two H\"older-continuous semiconjugacies $\pi:\Sigma\to M$ and $\hat\pi:\hat\Sigma\to M$ are defined by some shadowing properties. (Contrarily to $\hat\pi|\hat\Sigma^\#$, the map $\pi|\Sigma^\#$ is not finite-to-one.)

To prove the last claim of Theorem~\ref{theorem-SurfDiffCoding}, we use a strengthening of property (P4) above: we can replace $\chi/2$ by any number less than $\chi$, at least in the case of periodic orbits. We freely use the terminology and notations from \cite{Sarig-JAMS}, including the two semiconjugacies $\hat\pi:\hat\Sigma\to M$ and $\pi:\Sigma\to M$.

\begin{lemma}\label{lem-barchi1}
Given $\tilde\chi<\chi$, there is a coding $\pih:\hat\Sigma\to M$  with (P1)-(P4) as above that additionally satisfies the following property: for any  periodic sequence $\hat x\in\hat\Sigma$, $\hat\pi(\hat x)$ is $\tilde\chi$-hyperbolic.
\end{lemma}

\begin{proof}
Let $x:=\hat\pi(\hat x)$. Lemma 12.2 from \cite{Sarig-JAMS} yields a sequence of double charts $(\Psi_n)_{n\in\ZZ}\in\Sigma$ such that, for all $n\geq0$, ${}_{-n}[\hat x_{-n}\dots\hat x_n] \subset Z_{-n}(\Psi_{-n},\dots,\Psi_n)$. It follows from Proposition 4.11 in \cite{Sarig-JAMS} that  all points in $\pi([\Psi_0]\cap\Sigma^\#)$ can be written $\Psi_0(t)$ with $t\in\RR^2$ close to $0\in\RR^2$. Hence  $\hat\pi(\hat x)$ lift to $t_0$ in the domain of the chart $\Psi_0$: that is, for each $n\in\ZZ$, $t_n:=\Psi_n^{-1}\circ f^n(\hat\pi(\hat x))$ is well-defined.  Letting $f_k:=\Psi_k^{-1}\circ f\circ\Psi_{k-1}$, we have $t_k=f_k(t_{k-1})$ and:
 $$
     Df^n(x) = D\Psi_n\circ Df_{n}\circ\dots \circ Df_{1}\circ D\Psi_0^{-1}(x).
  $$
Thus
 $$
    \|Df^n(x)\| \geq \| D\Psi_n(t_{n-1})^{-1}\|^{-1} \cdot \|D\Psi_0(t_0)\|^{-1} \cdot \|Df_{n}\circ\dots Df_{1}\|
 $$
By Proposition 3.4 of \cite{Sarig-JAMS}, choosing the parameter $\epsilon>0$ of Sarig's construction small enough and considering vectors in the unstable cone, we obtain:
 $$
    \|Df_{n}\circ\dots \circ Df_{1}\|\geq (e^{\chi}-2\eps)^n\geq  e^{n\tilde \chi}.
  $$
Since $\Psi_k=\exp_{x_k}\circ\,C_\chi(x_k)$ where $x_k$ is the center of $\Psi_k$ \cite[eq.~(2.2)]{Sarig-JAMS}, we have: $\|D\Psi_n^{-1}\|\leq C_0 \cdot \|C_\chi(x_n)^{-1}\|$  for   some constant $C_0$ (depending only on $f$). 
 Since $\hat x$ is periodic, \cite[Theorem~10.2]{Sarig-JAMS} shows that $\Psi_n$ takes only finitely many values as $n$ ranges over $\ZZ$. 
 It follows that {setting $C_1(x):=\inf_{n\geq0} \lip(\Psi_n^{-1})^{-1}.\lip(\Psi_0)^{-1}>0$, we get:}
  $$
    \forall n\geq0\;  \|Df^n(x)\| \geq C_1(x) e^{n\tilde \chi}.
   $$
Hence the periodic orbit $\cO(x)$ has a positive exponent larger than or equal to $\tilde\chi$. A symmetric argument shows that $\cO(x)$ is $\tilde\chi$-hyperbolic.
\end{proof}

\begin{proof}[Proof of Theorem~\ref{theorem-SurfDiffCoding}]
We consider Sarig's coding with the addition of the property from Lemma~\ref{lem-barchi1}. The previous discussion shows that our Main Theorem applies. It produces a new coding of the form $\pih\circ q$, with $q$ H\"older-continuous, and whose image can be smaller, but only by a null set. The new coding is easily seen to satisfy our claims.
\end{proof}

We turn to the counting of hyperbolic periodic orbits. This requires the following estimate in eq.~\eqref{eq-lower-bound-Xi}. It is folklore, but since we did not find a reference we will {deduce it from}  \cite[chap. 7]{Kitchens-book}, using freely its terminology and notations.\footnote{We note that a similar estimate was obtained, e.g., in \cite{Buzzi-QFT} but with a stronger assumption (the SPR property) and stronger conclusion (an error estimate).}

A measure maximizing the entropy (or: m.m.e.) of some Borel automorphism is an invariant Borel probability measure which realizes the supremum of the Kolmogorov-Sinai entropy over all invariant probability measures.

\begin{lemma}\label{lem-barchi2}
If $(X,\sigma)$ is an irreducible Markov which is positively recurrent (i.e., it has some m.m.e. and its entropy is finite) with period $p$, then: 
 \begin{equation}\label{eq-lower-bound-Xi}
  \lim_{\scriptsize\begin{array}{c} n\to\infty\\ p|n\end{array}} e^{-n h_\top(f)}\cdot |\{x\in X:  |\{\sigma^kx:k\in\ZZ\}|=n \}| \geq p. 
 \end{equation}
\end{lemma}

\begin{proof}
We freely use results and notations from Kitchens' book \cite{Kitchens-book}  and in particular the generating functions $L_{ab}(z)$ and $R_{ab}(z)$.
First suppose that $X$ is mixing (i.e., $p=1$). Fix  some symbol $a\in\cA$ occuring in $X$ and set $\lambda:=e^{h_\top(f)}$.  Since $X$ is recurrent, Theorem 7.1.18 implies:
 $$
   \lim_{n\to\infty} \lambda^{-n}|\{x\in  X:x_0=a,\; \sigma^nx=x\}| =\frac1{\mu(a)}
 $$
where $\mu(a):=(1/\lambda)L_{aa}'(1/\lambda)$. Since $X$ is positive recurrent, Lemma 7.1.21 yields $\mu(a) = \ell^{(a)}\cdot r^{(b)}$ where $\ell^{(a)}:= (L_{aj}(1/\lambda))_{j\in\cA}$ and $r^{(b)}:=(R_{jb}(1/\lambda))_{j\in\cA}$. We also have $L_{aa}(1/\lambda)=R_{aa}(1/\lambda)=1$ for all $a\in\cA$ as $X$ is recurrent (see the proof of Lemma 7.1.8, recalling that, by definition, $T_{aa}(1/\lambda)=\infty$ if and only if $X$ is recurrent) . Thus,
 $$
   \frac1{\mu(a)}=\frac{1}{\sum_{j\in\cA} L_{aj}(1/\lambda)R_{ja}(1/\lambda)} 
   = \frac{L_{aa}(1/\lambda)R_{aa}(1/\lambda)}
   {\sum_{j\in\cA} L_{aj}(1/\lambda)R_{ja}(1/\lambda)} .
 $$
 Now  Lemma 7.2.15, implies that $\nu([b])= \frac{\ell^{(a)}_b r^{(a)}_b}{\sum_j \ell^{(a)}_j r^{(a)}_j} $ for any $a,b\in\cA$. Thus,
 $$
     \lim_{n\to\infty} \lambda^{-n}|\{x\in  X:x_0=a,\; \sigma^nx=x\}|  = \nu([a]).
  $$
For $p>1$, the cyclic decomposition from \cite[p. 223]{Kitchens-book} yields:
 $$
   \lim_{n\to\infty} \lambda^{-n}|\{x\in X:x_0=a,\; \sigma^nx=x\}| = p\nu([a]).
 $$
 Using the decomposition:
 $$
    \{x\in X:x_0=a,\; \sigma^nx=x\} = \bigsqcup_{k|n} \{x\in X:x_0=a,\;  |\{\sigma^j(x):j\in\ZZ\}|=k\}
 $$
and noting that $k|n$ implies $k=n$ or $k\leq n/2$, we get:
 $$
   \lim_{n\to\infty} \lambda^{-n}|\{x\in X:x_0=a,\; |\{\sigma^j(x):j\in\ZZ\}|=n\}| = p\nu([a]).
 $$

Since $\nu(X)=1$, a routine argument shows eq.~\eqref{eq-lower-bound-Xi}. 
\end{proof}

We are going to obtain the following relation between  periodic points and measures maximizing the entropy:

\begin{theorem} \label{thm-lower-per}
Let $f\in\Diff^{1+\alpha}(M)$ where $M$ is a closed surface and $\alpha>0$. Assume that there are distinct ergodic measures maximizing the entropy: $\mu_1,\dots,\mu_r$ with periods $p_1,\dots,p_r\geq1$. Fix $\tilde\chi<h_\top(f)$. Then
 \begin{equation}\label{eq-lower-per}
      \liminf_{\scriptsize\begin{array}{c} n\to\infty\\ p_1,\dots,p_r|n\end{array}} e^{-n h_\top(f)}\cdot |\per_{\tilde\chi}(f,n)| \geq p_1+\dots p_r.
 \end{equation}
\end{theorem}

When $f$ is $C^\infty$ smooth, Newhouse's Theorem \cite{Newhouse-1989} shows that there is at least one m.m.e. If, additionally, $f$ is topologically mixing, \cite{BCS} shows that there is a m.m.e. with period equal to $1$. Therefore:

\begin{corollary}
In the setting of the above theorem, assuming additionally that $f$ is $C^\infty$ we obtain:
 \begin{itemize}
  \item for some integer $p\geq1$, $\liminf_{\scriptsize n\to\infty,p|n} e^{-n h_\top(f)}\cdot |\per_{\tilde\chi}(f,n)| \geq p$;
  \item if $f$ is topologically mixing,  $\liminf_{\scriptsize n\to\infty} e^{-n h_\top(f)}\cdot |\per_{\tilde\chi}(f,n)| \geq 1$.
 \end{itemize}
\end{corollary}

This implies Theorem~\ref{theorem-SurfDiffPeriodic}. 

\begin{proof}[Proof of Theorem~\ref{thm-lower-per}]
We fix $\tilde\chi<h_\top(f)$ and consider a coding $\pih:\hat\Sigma\to M$  as in Lemma~\ref{lem-barchi1}.

If $\mu_1,\dots,\mu_r\in\Proberg(f)$ are distinct  m.m.e.'s, Sarig's \cite{Sarig-JAMS} shows that each $\mu_i$ is isomorphic to the product of a Bernoulli scheme and a circular permutation of some order $p_i$. Being an m.m.e. is invariant under Borel conjugacy, hence $\hat\Sigma$ carries distinct m.m.e.'s $\nu_1,\dots,\nu_r$ with $\pi_*(\nu_i)=\mu_i$. 

By general results about Markov shifts and their m.m.e.'s \cite{GurevicSavchenko-1998}, $\hat\Sigma$ contains disjoint irreducible components $X_1,\dots,X_r$, where each $X_i$ carries a distinct m.m.e.\ $\nu_i$.  In particular, each $X_i$ is positive recurrent and has period\footnote{ The period of an irreducible Markov shift is the greatest common divisor of the periods of its periodic points.} equal to $p_i$. 
 
By Lemma~\ref{lem-barchi1}, the following map is well-defined and injective:
 $$
   \pi:\{x\in \hat\Sigma: |\{\sigma^j(x):j\in\ZZ\}|=n\} \to \per_{\tilde\chi}(f,n).
 $$
The claim \eqref{eq-lower-per} now follows from Lemma~\ref{lem-barchi2}.
\end{proof}

\section{An obstruction to H\"older-continuous coding}\label{s-obstruction}

We prove  Theorem~\ref{prop-Holder-opti} characterizing surface diffeomorphisms with H\"older-continuous symbolic dynamics.
Recall that a map $\pi:X\to M$ is H\"older-continuous  with some positive exponent $\alpha$ if there is a constant $C<\infty$ such that, for all $x,y\in X^\#$,
 $$
     d(\pi(x),\pi(y)) \leq C \exp \left(-\alpha\inf\{|n|:x_n\ne y_n\}\right). 
 $$

Let $f$ be a diffeomorphism  of a compact $d$-dimensional manifold $M$ and let  $\mu\in\Proberg(f)$.  Write its Lyapunov exponents as  $\lambda_1(f,\mu)>\dots> \lambda_u(f,\mu)>0\geq\lambda_{u+1}(f)>\dots>\lambda_r(f,\mu)$. This measure has saddle type if $\lambda_{u+1}<0$ and $0< u<d$. Let 
 $$
    \Probhyp(f):=\{\mu\in\Proberg(f):\mu\text{ is aperiodic and of saddle type}\}.
 $$
Recall that for $\mu\in\Probhyp(f)$,
 $
    \chi(\mu):=\min(\lambda_u(f,\mu),-\lambda_{u+1}(f,\mu)).
 $
Sarig's theorem \cite{Sarig-JAMS} and its higher dimensional generalization by Benovadia \cite{benovadia-coding} yield a global coding if $\chi(f):=\inf\{\chi(\mu):\mu\in\Proberg(f)\}$ is positive. We prove:

\begin{proposition}\label{prop-Holder-opti1}
Let  $f\in\Diff^{1+}(M)$ with $M$ a closed manifold. Let $(S,X)$  be a Markov shift and let $\pi:(S,X)\to (f,M)$ be   a semiconjugacy. Assume that $\pi$ is H\"older-continuous with exponent $\alpha>0$. 
Given any $\nu\in\Proberg(\Sigma)$, if $\mu:=\pi_*(\nu)$ is hyperbolic and atomless, then:
 $$
    \hchi(\mu):=\min(\lambda_1(f,\mu), -\lambda_d(f,\mu)) \geq\alpha.
 $$
\end{proposition} 

This shows that $\inf\{\hchi(f,\mu):\mu\in\Probhyp(f)\}\geq\alpha$ is a necessary condition for the existence of a H\"older-continuous coding with exponent $\alpha$. Since $\hchi(f,\mu)=\chi(f,\mu)$ in dimension $2$, Theorem~\ref{prop-Holder-opti} is established.

\begin{proof}
Let $\nu\in\Proberg(\Sigma)$ such that $\pi_*(\nu)\in\Probhyp(f)$. For $\mu$-a.e. $x\in\Sigma$, the Pesin stable manifold of $y:=\pi(x)$ 
 $$
     W^s(y):=\{z\in M: \lim_{n\to\infty}\frac1n\log d(f^n(z),f^n(y)) <0 \}
 $$
satisfies:
  \begin{equation}\label{eq-Pesin}
       W^s(y)=\{z\in M: z=y \text{ or }  \lim_{n\to\infty}\frac1n\log d(f^n(z),f^n(y)) \in [\lambda_r(f,\mu),0)  \}.  
  \end{equation}
Since $\nu$ is not carried by a periodic orbit, it is carried by a nontrivial irreducible component of the Markov shift. Hence there is $z\in\Sigma$ such that $z\ne x$ and $z_n=y_n$ for all $n\geq0$. Therefore $d(\sigma^nx,\sigma^nz)=Ce^{-n}$ for some $C>0$ and all $n\geq0$. Now the H\"older-continuity of $\pi$ gives $C'>0$ such that
 $$
   \forall n\geq0\quad d(f^n(y),f^n(\pi(z))) \leq C' e^{-\alpha n}.
 $$
This exponential convergence implies that $\pi(z)\in W^s(y)$. 
By eq.~\eqref{eq-Pesin}, $\lambda_r(f,\pi_*(\nu))\leq-\alpha$. By considering $(f^{-1},\mu)$, we obtain $\lambda_1(f,\pi_*(\nu))\geq\alpha$. Thus $\hchi(\pi_*(\nu))\geq\alpha$.

\end{proof}

\appendix
\section{Further remarks}\label{appendix-Bowen}

\subsection{Canonical Bowen relation}
A semiconjugacy  $\pi:X\to Y$ of a symbolic system $X$ can admit several Bowen relations.  However, one can define its  \emph{canonical relation} over its alphabet $\cA$ by:
 $
   \forall a,b\in\cA\quad  a\stackrel{\pi}{\sim} b \stackrel{\rm def}\iff \pi([a]_X)\cap\pi([b]_X)\ne\emptyset
 $
(recall that $[\cdot]_Z$ denotes the cylinder in $Z$ defined by some word).
The above relation is obviously reflexive and symmetric. We denote by $\BEpi$ the induced Bowen equivalence on $X$.

\begin{lemma}
For an arbitrary semiconjugacy $\pi:X\to Y$, the following implication holds:
 \begin{equation}\label{eq-piImpliesSim}
   \forall x,y\in X\quad \pi(x)=\pi(y)\implies x\BEpi y.
 \end{equation}
If the semiconjugacy $\pi$ is Bowen, then the canonical relation is a Bowen relation and it is the minimal one: if $\sim$ is any Bowen relation for $\pi$, then
 $a \stackrel{\pi}{\sim} b\implies a\sim b$ for any $a,b\in\cA$.
\end{lemma}

\begin{proof}
The implication \eqref{eq-piImpliesSim} is immediate. Now assume that $\pi$ has some Bowen relation $\sim$ and let $a,b\in\cA$ with $a\stackrel\pi\sim b$: there are $x\in[a]_X$ and $y\in[b]_X$ with $\pi(x)=\pi(y)$. The Bowen property for $\sim$ gives  $a\sim b$ so we have proved $a\simpi b\implies a\sim b$. Now it is obvious that $\simpi$ is a Bowen relation for $\pi$.
\end{proof}

\begin{remark}
We do not know whether the reflexive and symmetric relation that appears in Sarig's construction (called affiliation) is canonical. Additionally, we do not know if the Bowen quotient (Theorem~\ref{thm-Bowen-quotient}) of a canonical relation is itself  canonical. 
\end{remark}

\subsection{Consequences for continuous extensions}
In our most important examples, the semiconjugacy is continuous over the Markov shift $X$ but the Bowen property is only known for the \emph{regular part} $X^\#$. It is then natural to consider $\pi|\overline{X^\#}$. It is easy to see that $\overline{X^\#}$ is a Markov shift: setting $\cA^\#:=\{x_0:x\in X^\#\}\subset\cA$,
  $  \overline{X^\#} = X \cap (\cA^\#)^\ZZ.$
As $\pi$ is continuous,  $\piS|\overline{X^\#}$ is determined by its regular part but the Bowen property may fail to extend to $\overline{X^\#}$. Denote by $\stackrel{\pi}{\sim}$ the canonical relation induced by $\piS|X^\#$ and by $\BEpi$ the corresponding relation on $\overline{X^\#}$.

\begin{lemma}
Let $\pi:X\to Y$ be a continuous semiconjugacy with $X$ a Markov shift. If the restriction of $\piS$ to $X^\#$ has the Bowen property, then:
 \begin{equation}\label{eq-closure-Bowen}
       \forall x,y\in \overline{X^\#}\quad x\BEpi y \implies \pi(x)=\pi(y).
 \end{equation}
\end{lemma}

\begin{proof}
Let $x,y\in \overline{X^\#}$ with $x\BEpi y$ and $n\geq1$.  As $x_{-n}\simpi y_{-n}$, there are $x^{-n}\in\sigma^n[x_{-n}]_{X^\#}$, $y^{-n}\in\sigma^n[y_{-n}]_{X^\#}$ with $\pi(x^{-n})=\pi(y^{-n})$. By the Bowen property, this implies $x^{-n}\BEpi y^{-n}$. Likewise, there are $x^n\in\sigma^{-n}[x_n]_{X^\#}$, $y^n\in\sigma^{-n}[y_n]_{X^\#}$ with $x^n\BEpi y^n$. Define $\tilde x^n\in X$ by:
 $$
    \tilde x^n_{k}=\alter{
    x^{-n}_{k} &\text{ for }k\leq -n\\
    x_k        &\text{ for }|k|\leq n\\
    x^n_k &\text{ for }k\geq n.}
 $$
 Define $\tilde y^n$ similarly. Observe that $\tilde x^n,\tilde y^n$ both belong to $X^\#$ and $\tilde x^n\BEpi\tilde y^n$ so that $\pi(\tilde x^n)=\pi(\tilde y^n)$. Since $\pi$ is continuous, $\pi(x)=\lim_n \pi(\tilde x^n)=\lim_n \pi(\tilde y^n)=\pi(y)$.
\end{proof}

Still, the Bowen property may fail to extend to $\overline{X^\#}$ as in the following example.

{\footnotesize
\begin{example}
Consider the graph with set of vertices
$\ZZ\cup\{\alpha,\omega\}$ and arrows $n\to(n+1)$, $\alpha\to n$, $n\to\omega$, $\alpha\to\alpha$, $\omega\to\omega$ {\rm(}for all $n\in\ZZ${\rm)}. Let $(S,X)$ be the induced Markov shift. Let $\pi:X\to Y\subset \{0,1,1/2,\dots\}^\ZZ$ be the semiconjugacy  such that {\rm(}the vertical bar is immediately to the left of index $0${\rm)}:
 \begin{enumerate}
  \item $\alpha^\infty$, $\omega^\infty\mapsto 0^\infty$;
  \item $\alpha^\infty| n\cdot(n+1)\cdots (n+\ell-1)\cdot\omega^\infty\mapsto 0^\infty| \tfrac1\ell\cdot\tfrac1\ell\cdots\tfrac1\ell\cdot0^\infty$ (for all $n\in\ZZ$, $\ell\in\NN$);
  \item $\alpha^\infty \cdot n\cdot(n+1)\cdot(n+2)\dots\mapsto 0^\infty$ (for all $n\in\ZZ$);
  \item $\dots(n-2)\cdot(n-1)\cdot n\,\omega^\infty\mapsto 0^\infty$ (for all $n\in\ZZ$);
  \item $\dots -2\cdot-1\cdot0\cdot1\cdot2\dots\mapsto 0^\infty$.
 \end{enumerate}
 $\pi$ is well-defined and continuous. The regular sequences are those in (1) and (2).

It is easy to check that $\pi$ is Bowen on $X^\#$ for the symmetric relation generated by $n\sim m$ for all $n,m\in\ZZ$ and $\alpha\sim\omega$. If the semiconjugacy~$\pi$ was Bowen on $X$, (1) and (5) would imply that all symbols would be related, contradicting (2).
\end{example}
}

\section{A locally compact recoding}\label{appendix-Q}

Our Main Theorem does not preserve local compactness.  In this appendix we provide an alternate construction which preserves local compactness at the expense of a slightly weaker injectivity property:

\begin{theorem}\label{maintheorem-Q}
Let  $(S,X)$ be a \emph{locally compact} Markov shift on some alphabet $\cA$. Let $X^\#$ be its regular part.
Let $\pi:(S,X^\#)\to(T,Y)$ be a Borel semiconjugacy such that:
\begin{itemize}
  \item[-] $(T,Y)$ is a Borel automorphism;
  \item[-] $\pi$ is finite-to-one, i.e., $\pi^{-1}(y)$ is finite for every $y\in Y$;
  \item[-] $\pi$ has the Bowen property with respect to a locally finite relation on $\cA$.
\end{itemize}   
Then there are a \emph{locally compact} Markov shift $(\tilde S,\tilde X)$ and a $1$-Lipschitz map $\phi:\tilde X\to X^\#$ such that $\pi\circ\phi:\tilde X\to Y$ defines a semiconjugacy satisfying:
\begin{itemize}
  \item[-] $\pi\circ\phi|\tilde X^\#$ is injective;
  \item[-] $\pi\circ\phi(\tilde X^\#)$ carries all invariant measures of $\pi(X^\#)$. 
\end{itemize}
\end{theorem}

\newcommand\cV{\mathcal V}

\begin{proof}
We explain the required changes in proof of the Main Theorem. 
An inspection of the proof of the Main Theorem shows that the local compactness is lost in Lemma~\ref{lem-magic-subset}. It suffices to replace this lemma with the following statement.
\end{proof}

\begin{lemma}\label{lem-magic-subset-Q}
Let $X$ be a locally compact Markov shift. Let $\cW:=(W^j)_{1\leq j<J}$ with $1<J\leq \infty$ be an enumeration of $X$-words. Then there is a one-block code $p:S\to X$ defined on a  Markov shift $S$ such that:
 \begin{enumerate}
  \item $S$ is locally compact;
  \item $p|S^\#$ is injective;
  \item the image of $p|S^\#$ is $X_\cW\setminus N$, where $X_\cW$ is the of sequences in $X$ which see i.o. some word from $\cW$ and $N$ is a null set.
  \end{enumerate} 
\end{lemma}

\begin{proof}[Proof of the lemma]
Fix some $1\leq i<J$ and let $N:=|W^i|$. Recall the graph $\cG^3$ and the extension $p^i:S_i\to X_i$ defined in the proof of Lemma~\ref{lem-magic-subset}. The vertices of $\cG^3$ are couples $(v,j)$ with $v$ a $X^{[N]}$-word with $v_0=W^i$ and $j$ an integer such that $0\leq j <|v|$. The set of such vertices $(v,j)$ is $\cV^3$. Except in somewhat trivial situations, the lengths of the words $v$'s from $\cV^3$ are unbounded. 

Note that only vertices $(v,j)$ with $j=|v|-1$ can have outdegree larger than $1$ and that only vertices $(v,0)$ can have indegree large than $1$. However $(v,|v|-1)\to(w,0)$ whenever $v_{|v|-1}\to w_0=W^i$ in $X^{[N]}$. Since they are infinitely many words $v$ (their length being unbounded), $S_i$ is not locally compact.

We define a new graph $\cG^4$ as follows. Let:
 $$\begin{aligned}
   &\cV^4:=\{ (v,j,L_-,L_+)\in\cL(X^{[N]})\times\NN\times\mathbb N: (v,j)\in\cV^3\text{ with } |v|\leq \min(L_-,L_+)\};\\
   &(v,j,L_-,L_+)\to(w,k,M_-,M_+) \text{ if and only if } (v,j)\stackrel{\cG^3}\to (w,k)\\
   &\qquad\qquad\qquad \text{ and }  M_-:=\max(|v|,L_--1),\; L_+:=\max(|w|,M_+-1).
  \end{aligned}$$
Let $T_i$ be the Markov shift defined by $\cG^4$ and define $p|T_i$ as $p^i\circ q$ where $q(v,k,L)=(v,k)$.  
  
\step{1}{Local compactness.}
Let $(v,j,L_-,L_+) \to (w,k,M_-,M_+)$ on $\cG^3$. Note that 
 $$
    M_-=\max(|v|,L_--1)\text{ and }M_+\leq L_++1.
 $$
It follows that, given $(v,j,L_-,L_+)$ there are finitely many possibilities for $(M_-,M_+)$. In particular $|w|$ is bounded. Now $w=v$ or $w$ starts by the fixed $X$-word $W^i$. Since $X$ is locally compact, this gives finitely many possibilities for $(w,k)$. 
Thus the outdegree of any vertex in $\cG^3$ is finite. Likewise the indegree of any vertex is finite.  The local compactness, i.e., item (1), is proved.

\medbreak

We define the \emph{flat part} of $S^i=\Sigma(\cG^3)$ to be:
 $$
 S_i^\flat:=\{(v,j)\in S^i:\lim_{n\to\pm\infty} |v^n|-|n|=-\infty \}
  \text{ where } (v^n,j^n)_{n\in\ZZ}=(v,j).
 $$

\step{2}{The flat part has full measure for any invariant probability measure $\mu$ on $S_i$.}

We can restrict to $\mu$ ergodic. Now, assume by contradiction that $\lim_{n\to\pm\infty} |v^n|-|n|=-\infty$ fails for  a set $D$ of points $(v,j)\in S_i$ with positive $\mu$-measure. Hence, for any $(v,j)\in D$, there are a constant $C>0$ and arbitrarily large integers $n$ such that $|v^n|\geq |n|-C\geq |n|/2$. We assume that one can choose these integers to be  positive, the negative case being similar. For such an integer $n$, if  $\sigma^k(v,j)\in E_n:=\{(w,\ell)\in S^i:|w^0|>n\}$ for $k=k_0$ for some $0\leq k<n$, then it holds for all $k$ in some positive interval segment of length $|v^{k_0}|\geq n/2$ and containing $k_0$. Therefore it holds for at least $n/2$ integers  $0\leq k<(3/2)n$. Hence, for any integer $N$,
$$
  \forall x\in D\quad  \limsup_{n\to\infty} \frac1{(3/2)n} \#\{0\leq k<(3/2)n:\sigma^k(x)\in E_N\}\geq1/4.
$$
By the pointwise ergodic theorem, this implies that $\mu(E_N)\geq 1/4$ for all $N$, contradicting the $\sigma$-additivity of $\mu$. Hence  $\mu(S_i^\flat)=1$.

\step{3}{There is a canonical lift $\iota:S_i^\flat\to T_i$ which is well-defined with $q\circ\iota=\id$.}

Given $(v,j)\in S^i$, we let:
 $$
    L_-^n(v):= \max_{k\geq0} |v^{n-k}|-k \text{ and }
    L_+^n(v):= \max_{k\geq0} |v^{n+k}|-k.
 $$
and define the \emph{canonical lift} as:
 $$
  \iota: (v^n,j^n)_{n\in\ZZ} \longmapsto (v^n,j^n,L_-^n(v),L_+^n(v))_{n\in\ZZ}
 $$

Let $(v,j)\in S_i^\flat$. We check that $\iota(v,j)$ is well-defined.
First, the numbers $L_-^n(v), L_+^n(v)$ are well-defined since, by the definition of $S_i^\flat$, $|v^{n+k}|-k<0$ for all large $k\geq0$. Note also:
 $$\begin{aligned}
    L_+^n(v) &=  \max(|v^n|,\max_{k\geq1} |v^{n+k}|-k) = \max(|v^n|,\max_{k\geq0} |v^{n+1+k}|-k-1)\\ &= \max(|v^n|,L^+_{n+1}(v)-1)
 \end{aligned}$$
 and likewise $L_-^{n+1}(v)=\max(|v^{n+1}|,L_-^n(v)-1)$. These identities show that:
  $$
   (v^n,j^n,L_-^n(x),L_+^n(v))_{n\in\ZZ}\in T_i.
  $$
Thus $\iota:S_i^\flat\to T_i$ is well-defined. 

\medbreak

The identity $q\circ \iota=\id$ is trivial. 

\step{4}{The map $q:T_i^\#\to S_i^\flat$ is well-defined and $\iota\circ q|T_i^\#=\id$.}

To see that $q$ is well-defined, it suffices to check that $q(T_i^\#)\subset S_i^\flat$. Let $z:=(v,j,L_-,L_+)\in T_i^\#$. Thus there is some $z^*:=(v^*,j^*,L_-^*,L_+^*)\in\cV^4$ that appears infinitely many times in $z^n$ when $n\geq0$. Let $n\geq0$ be a large integer. Let $m(n)$ be the largest index less than $n$ such that $z^{m(n)}=z^*$. Observe that $L_+^n\leq L_+^*+(n-m)$. Thus $|v^n|-n\leq L_+^*-m(m)$. Thus $\lim_{n\to\infty} |v^n|-n= -\infty$ as $\lim_{n\to\infty} m(n)=+\infty$. The limit when $n\to-\infty$ is handled similarly using $L_-^*$, proving that $q$ is well-defined.

We turn to the identity $\iota\circ q|T_i^\#=\id$. Let $x\in T_i^\#$. We must show that it coincides with the canonical lift $\tilde x:=\iota(q(x))$. Write $(v^n,j^n,L_-^n,L_+^n):=x_n$ and $(\tilde v^n,\tilde j^n,\tilde L_-^n,\tilde L_+^n):=\tilde x_n$. 

Since $x\in T_i^\#$, there is a symbol $a:=(v,j,M_-,M_+)$ which appears infinitely often in the past of $x$. Thus there are arbitrarily large integers $N$ such that $x_{-N}=a$. It follows that $L_-^K=|v^K|$ for some $-N\leq K\leq -N+M_-+1$. Indeed, otherwise one would have: $L_-^{-N+M_-+1}=M_- - M_- - 1 < 0$, a contradiction.

By an easy induction, the definition of the arrows in $\cG^4$ implies that:
 $$
    \forall n\in\ZZ\quad  L_-^n \geq \max_{k\geq0} |v^{n-k}|-k
  $$
 It follows that the canonical lift is as small as possible in the following sense:
 $$
      \forall n\in\ZZ\quad \tilde L_-^n\leq L_-^n.
    $$
Since $L_-^K=|v^K|$, then $L_-^K=\tilde L_-^K$ and therefore $L_-^k=\tilde L_-^k$ for all $k\geq K$ and in particular, all $k> -N+M_-$.  Since $N$ is arbitrarily large, it follows that $L_-^n=\tilde L_-^n$ for all $n\in\ZZ$. A similar reasoning applies to the sequence $(L_+^n)_{n\in\ZZ}$, concluding the proof that $x=y$ and therefore of the identity.

We note that this identity implies that $q$ is injective. The theorem is proved.
 \end{proof}

\section{Application to Sina\"\i\ billiards collision maps}

We prove Theorem~\ref{theorem-BilliardMap}, i.e., the lower bound on the periodic points for the billiard maps considered by Baladi and Demers \cite{Baladi-Demers-2018}. We fix such a collision map $T_B$  defined by a two-dimensional Sina\"\i\ billiard satisfying conditions (BD1) and (BD2) quoted in our introduction.  Theorem~2.4 of \cite{Baladi-Demers-2018} yields a strongly mixing measure $\mu_*\in\Proberg(T_B)$ such that
 \begin{equation}\label{eq-TB-entropy}
     h(T_B,\mu_*) = \sup \{h(T_B,\nu): \nu\in\Prob(T_B)\} = h_* 
 \end{equation}
 where $h_*$ is a combinatorial entropy from their eq. (1.1).
 
Let $\pi:(\Sigma,\sigma)\to(M,T_B)$ be the coding built by \cite[Thm. 1.3]{Lima-Matheus-2016} for some hyperbolicity parameter $\Lambda<\chi<\chi(f,\mu_*):=\min(\lambda^1(f,\mu_*),-\lambda^2(f,\mu_*))$. As in Sarig's construction for diffeomorphisms, $\Sigma$ is a Markov shift  and $\pi$ is a H\"older-continuous semiconjugacy. Note that $M$ is a two-dimensional compact manifold with boundary and that, writing $M_1$ for the domain where both $T_B$ and its inverse are well-defined and differentiable, 
  $$
      \pi(\Sigma^\#) \subset \pi(\Sigma) \subset \bigcap_{n\in\ZZ} T_B^{-n}(M_1) \subset M\setminus\partial M
   $$
(the middle inclusion is nontrivial but is proved in \cite{Lima-Matheus-2016}). 

An inspection of the proof in \cite{Lima-Matheus-2016} shows that the semiconjugacy on $\Sigma^\#$ admits a  Bowen relation  just as in the smooth case of \cite{Sarig-JAMS}. Indeed, though this is not stated in \cite{Lima-Matheus-2016}, it follows from the same arguments as in the original case, see Section~\ref{s-applications}. This Bowen relation is locally finite thanks to \cite[Prop. 7.1(2)]{Lima-Matheus-2016}.

According to \cite[Thm. 2.4]{Baladi-Demers-2018}, $\mu_*$ is $T$-adapted in the sense of \cite{Lima-Matheus-2016}. Since it is $\chi$-hyperbolic from the choice of $\chi$, \cite[Thm. 1.3]{Lima-Matheus-2016} implies the existence of $\hat\mu_*\in\Proberg(\sigma)$ such that $\pi_*(\hat\mu_*)=\mu_*$. Since $\pi|\Sigma^\#$ is finite-to-one, $\pi_*:\Prob(\Sigma,\sigma)\to\Prob(T_B)$ preserves the entropy. Therefore, eq.~\eqref{eq-TB-entropy} implies that $\hat\mu_*$ is an m.m.e. for $\Sigma$.

As in Lemma~\ref{lem-barchi1}, we can assume that the image of any periodic orbit of $\Sigma$ is $\chi'$-hyperbolic for any $\chi'<\chi$. In particular, this holds for $\chi'=\Lambda$.

Our Main Theorem now allows us to replace $\pi:(\Sigma,\sigma)\to M$ by an injective coding still denoted by $\pi:\Sigma\to M$. The lift $\hat\mu$ by $\pi$ of the m.m.e. $\mu_*$ is now measure-preservingly isomorphic to $\mu_*$, hence $\hat\mu_*$ is also strongly mixing. This implies that the irreducible component of $\Sigma$ carrying $\hat\mu_*$ has period $1$.

Finally, Lemma \ref{lem-barchi2} gives the claimed lower bound on the periodic orbits.

\bibliographystyle{plain}

\end{document}